\documentclass[9pt,english,twosided]{article}
\usepackage[english,activeacute]{babel}
\usepackage{amsmath,amsthm}
\usepackage{amscd}
\usepackage{amssymb,latexsym}
\usepackage{enumerate}
\usepackage{amsmath}
\usepackage{amsfonts}
\usepackage{amssymb}
\usepackage{lscape}
\usepackage{color}

\title{Extented Paley-Wiener type theorems for the Mellin and Fourier transforms on the half-real line}
\author{C\'esar Augusto del Corral Martinez\\
        ce-del@uniandes.edu.co\\
        cesar.math@gmail.com}
\date{ Universidad de los Andes\\
[0.5cm] 2016}

\newtheorem{thm}{Theorem}

\newtheorem{lemma}{Lemma}
\newtheorem{defin}{Definition}
\newtheorem{prop}{Proposition}

\newtheorem{cor}{Corollary}

\newtheorem{remark}{Remark}
\newtheorem{example}{Example}

\newcommand{\C}{\mathbb{C}}
\newcommand{\R}{\mathbb{R}}

\newcommand{\N}{\mathbb{N}}
\newcommand{\uP}{\underline{\mathbb{P}}}
\newcommand{\oP}{\overline{\mathbb{P}}}

\newcommand{\PP}{\textsf{p}}

\newcommand{\Z}{\mathbb{Z}}
\newcommand{\dlim}{\small{\mathcal{D}^\prime} \!\text{-}\!\lim}

\newcommand{\la}{\langle} 
\newcommand{\ra}{\rangle} 

\newcommand{\normd}{{d\hspace{-0,05cm}\bar{}\hspace{0,05cm}}}

 \DeclareMathOperator{\Res}{Res}

\begin{document}
\maketitle

\section*{Abstract}

In this paper we present generalisations of  Paley-Wiener type theorems  to  Mellin and (Laplace-)Fourier transforms  of  rapidly decreasing smooth functions with positive support and log-polyhomogeneous asymptotic expansion at zero. This article is based on  the thesis \cite{CdC} uses  results  borrowed from  \cite{RS}, \cite{RS81}, \cite{P79}, \cite{FGD}, \cite{P79}, \cite{SZ}.
\vspace{1cm}

\section*{Introduction} 

A classical theorem  due to R. Paley and N. Wiener \cite{PW} provides a necessary and sufficient condition to  build a holomorphic  extension of the Fourier transform $\widehat{f}$ of a function $f\in L^2(\R)$ with exponential growth $|\widehat{f}(\zeta)|\leq Ce^{a|\zeta|}$  and supported on   $ [-a,a]$.   This theorem was generalized to  distributions  by L. Schwartz in  \cite{Schwartz68}, and since then this type of results are known as Paley-Wiener theorems.\\ 
\\
The main  goal of this paper is to provide Paley-Wiener type theorems for the Mellin and Fourier transform of rapidly decreasing functions with positive support and log-polyhomogeneous asymptotic behavior at zero. To  the author's knowledge, in that degree of generalization, the results concerning  the Fourier transform included in Section \ref{SS:P-W.Fourier.Log-poly} are new. As a particular case we recover a Paley-Wiener type theorem for rapidly decreasing functions which are smooth up to the boundary. This result plays an important role in  Boutet de Monvel's pseudo-differential calculus \cite{BdM}, and  provides a fundamental tool for the traciality of the non-commutative residue on BdM's algebra done by  B.V. Fedosov, F. Golse, E. Leichtnam and E. Schrohe  in \cite{FGLS}. 
The tools presented here are used in an essential way in  \cite{CdC}  to prove the traciality of the so called {\em canonical trace}  for log-polyhomogeneous pseudo-differential operators (see \cite{K-V} and \cite{Lesch99}) on manifolds with boundary. We hope that this short survey can also be of use for whom might want to get acquainted  with Paley-Wiener type theorems, independently of their applications in geometric analysis.
\\ \\
In order to make this document self-contained  we  also  recall well-known Paley-Wiener type theorems.  As we go along we compare   our results with  existing results in the literature, in particular, we compare a Paley-Wiener result involving the Mellin tranform of rapidly decreasing functions with positive support and log-polyhomogeneous asymptotic behavior at zero   (cf. Theorem \ref{Type Paley-Wiener.thm's}, see also \cite{RS}, \cite{RS81})   with the respective Fourier transform (cf. Theorem \ref{P-W.Mellin.Asymp.Log-poly}).
\\  \\
This document is organized as follows: In Section \ref{P-W.mellin.smooth}, we recall  basic properties  of and Paley-Wiener type results for the Mellin transform   of smooth compactly supported functions $C^\infty_c(\R_+)$  as well as of compactly supported functions in $L^2(\R_+)$ (see, \cite{RS81}, \cite{P79}, \cite{FGD}). Finally, we give similar results for $\mathcal{S}(\overline{\R}_+)$ (cf. Proposition \ref{Smooth.P-W.Mellin}). In Section \ref{SS:P-W.Mellin.log-poly}.  Using  results  of  \cite{FGD}, we deduce a Paley-Wiener type theorem for functions in $\mathcal{S}_{\PP}(\R_+)$ (cf. Theorem \ref{P-W.Mellin.Asymp.Log-poly}). In Section \ref{Subsec-extend-distr}, following \cite{P79} we recall a Paley-Wiener type theorem of the Mellin trasnform of extendable compactly supported distribution $\mathcal{E}'_+(\R_+)$ (cf. Proposition \ref{Mellin.Dist.}).\\
In Section \ref{SS:HomoD}, we derive an explicit expression for the Fourier transform of log-homogeneous distributions, restricting  ourselves to distributions  of order $a\in\C\setminus\{\dots,-2, -1\}$ and log-type $k\in\Z_{\geq 0}$. In Section \ref{SS:P-W.Fourier.Log-poly}, we  combine  the results obtained in Section \ref{SS:HomoD}  to derive  a Paley-Wiener type theorem for the Fourier transform of functions in $\mathcal{S}_{\PP}(\R_+)$. More precisely, we relate the Fourier transform of a function in $\mathcal{S}_{\PP}(\R_+)$ with a holomorphic function in $\C_-:=\{\zeta\in\C \ |\ \textup{Im}\, \zeta<0\}$ with log-polyhomogeneous asymptotic behavior at infinity, which is  continuous up to the boundary  in $\mathcal{S}'(\R_+)$ (cf. Theorem \ref{Type Paley-Wiener.thm's}). In particular, this result can be applied to functions in $\mathcal{S}(\overline{\R}_+)$ (cf. Proposition \ref{Paley-WienerProposition}). Finally, in Section \ref{SS:WP-thm.tempred.distributions} we recall a Paley-Wiener type result of the Laplace-Fourier transform of tempered distribution with compact support, see \cite{SZ}, similar to the one described in Section \ref{Subsec-extend-distr} for the Mellin transform.

\section{Notations and preliminary definitions}

We introduce the necessary notations for the subsequent sections and the definitions to follow.\\  \\
\textbf{Notations:} 
\begin{itemize}
\item $\R_+:=\{x\,|\,\,  x>0\}$  resp. $\R_-:=\{x\, |\ x<0\})$ and $\overline{\R}_+:=\{x\, |\, x\geq 0\}$  resp. $\overline{\R}_-:=\{x\, |\, x\leq 0\})$, denote the real half-spaces without and with boundary, respectively. Similarly, 
$$\C_+:=\{\zeta\in\C \ |\ \textup{Im}\, \zeta>0\}\ \ (\text{resp. } \C_-:=\{\zeta\in\C \ |\ \textup{Im}\, \zeta<0\})\text{  and }$$
$$\overline{\C}_+:=\{\zeta\in\C \ |\ \textup{Im}\, \zeta \geq 0\}\ \  (\text{resp. }\overline{\C}_+:=\{\zeta\in\C \ |\ \textup{Im}\, \zeta\leq 0\})
$$
denotes  the complex half-plane\index{complex half-plane} without and with boundary, respectively.\\
\item By a {\em cut-off function}  $\omega$ around zero we mean a smooth compactly supported function which is nonnegative, decreasing, and equal to  $1$ near  zero.  By an {\em excision function}\index{excision function} $\chi$ we mean a smooth function which is nonnegative and vanishes near zero and is equal to $1$ outside of a neighborhood of zero. Note that $1-\chi$ is a cut-off function.
\item 
We denote by $\xi\mapsto [\xi]$ the strictly positive function for which $[\xi]=|\xi|$ for $|\xi|\geq 1$,  and by $O([\xi]^{-\infty})$\index{$O([\xi]^{-\infty})$} we mean a {\em rapidly decreasing function}\index{rapidly decreases function}, i.e. a   function that decreases  faster than any polynomial.
\item Let $U$ be an open subset of $\R$. Let $C^\infty_c(U)$ be the set of all compact support smooth functions on $\R$. Let $\mathcal{S}(\R)$ be the set of all smooth rapidly decreasing function on $\R$, i.e., $u$ lies in $\mathcal{S}(\R)$ iff $u$ satisfies that for any $\alpha,\beta\in\{0,1,2,\dots\}$, there exists a positive constant $C_{\alpha,\beta}$ such that $$\sup_{x\in\R} |x^\alpha \partial^\beta_x u(x)|\leq C_{\alpha,\beta}.$$
\item Let $\mathcal{S}(\overline{\R}_+)$ be the set of Schwartz functions
on $\R_+$ which are {\em smooth up to the boundary}, i.e. $u$ lies in
$\mathcal{S}(\overline{\R}_+)$ if there exists a function $\tilde{u}\in \mathcal{S}(\R)$ such that its restriction $\tilde{u}(x)|_{x>0}$ coincides with $u(x)$. Let  $\mathcal{S}'(\R)$ be the space of {\em tempered distributions}.
\item 
Let $U\subset \R$ be an open subset. A distribution $f$  in $U$ is a linear form on $C^\infty(U)$ such that for every compact set $K\subset U$ there exist constants $C$ and $k$ such
\begin{equation}\label{Eq:2.1.2}
|f(\phi)|\leq C\sum_{|\alpha|\leq k}\sup |\partial^\alpha \phi|,\,\,\, \forall \phi\in C^\infty_c(K).
\end{equation} The set of all distributions on $U$ is denoted by $\mathcal{D}'(U)$. If the same integer $k$ can be used in  (\ref{Eq:2.1.2}) for every compact $K$ we say that $f$ is of order $\leq k$ and we denote the set of such distributions by $\mathcal{D}^{',k}(U)$ . 
\item We say $f\in\mathcal{D}'(U)$ vanishes on $V\subset U$ if $\la f,\phi \ra=0$ for all $\phi\in C^\infty_c(U)$ with support in  $V$.
 The maximal open set $V$ on which the distribution $f$ vanishes is called the {\em support} of $f$.
\item
The set of compactly supported distributions is denoted by $\mathcal{E}'(\R)$, and $\mathcal{E}'_+(\R_+)$ denotes the set of positive compact supported distribution in $\mathcal{E}'(\R)$ which can be extended to a compact support distribution in $\R$. 
\end{itemize}

We are now ready to introduce some basic definitions (see e.g. \cite{Rudin}, \cite{PW}, \cite{Schwartz52}, \cite{SB}). 

\begin{defin}
Let $U\subset \R$ be an open subset. Denote by $C^\infty_c(U)$ the set of all smooth functions with compact support in $U$.
\begin{itemize}
\item \textit{The Mellin Transform:}
The \textit{Mellin Transform}\index{Mellin transform} is defined as a map $\mathcal{M}:C_c^\infty(\R_+)\rightarrow \mathcal{A}(\C)$ given by
\begin{equation}\label{Mellin.integral}
\mathcal{M}[f](s):=\int_0^\infty t^sf(t)\frac{dt}{t}, \;\;\;  s\in\C,
\end{equation}
where $\mathcal{A}(\C)$ denotes the set of all holomorphic functions on $\C$. The Mellin map\index{Mellin map} $f\mapsto \mathcal{M}[f]$ is continuous, and has a continuous inverse $\mathcal{M}^{-1}$ given by
\begin{equation}\label{Mellin.Fourier.eq}
 f(t)=\int_{-\infty}^{\infty}t^{-(\xi+i\eta)} \mathcal{M}[f](\xi+i\eta)d\eta \;\;\;\; \text{ for arbitrary }\xi.
\end{equation}
which  defines a continuous map. 
\item \textit{The Fourier transform:} Let   $\mathcal{S}(\R)$ denote the set of all smooth rapidly decreasing functions, also called the set of Schwartz functions.\\
The \textit{Fourier transform} $\mathcal{F}:\mathcal{S}(\R)\to \mathcal{S}(\R)$ is a continuous bijection with continuous inverse given by
\begin{equation}\label{Eq:def.Fourier}  
\mathcal{F}[u](\xi):=\int_{\R} e^{-ix\xi}u(x)\,\normd x\,\,\,  \text{and}\,\,  \mathcal{F}^{-1}[u](x):=\int_{\R} e^{ix\xi}u(\xi)\,\normd \xi,
\end{equation}
where $\normd x:=\frac{1}{(2\pi)^{1/2} }dx$, $\mathcal{F}[u](\xi)$ is also denote by $\widehat{u}$.
\end{itemize}
\end{defin}

The following proposition relate the Mellin  and the Fourier transforms, a proof follows from straight way computation.

\begin{prop}  
\textit{The Mellin transform in terms of a Fourier transform:}
 For $\xi\in\R$, let $m_\xi:C_0^\infty(\R_+)\rightarrow C^\infty(\R)$ be the isomorphism of vector spaces given by $f\mapsto e^{-\xi t}f(e^{-t})$.  Then, for $s=\xi+i\eta$  we have
\begin{equation}\label{Fourier.Mellin.Rel}
 \mathcal{M}[f](s)=\mathcal{F}\circ m_\xi [f](\eta)= \mathcal{F}[e^{-\xi t}f(e^{-t})](\eta), \,\,\, \forall f\in C^\infty_c(\R_+).
\end{equation}
 If $\chi:\R\rightarrow \R_+$ denotes the diffeomorphism of groups $\chi(x)=e^{-x}$ then,
 its pullback  $\chi^\ast:C^\infty_c(\R_+)\rightarrow C^\infty_c(\R)$, given by $(\chi^\ast f)(x)=f(e^{-x})$,  coincides with $m_0$.
 In particular, for $s=i\eta$, we have
 $$\mathcal{M}[f](i\eta)=\int_0^\infty t^{i\eta}f(t)\frac{dt}{t}=\int_{-\infty}^{\infty}e^{-ix\eta}f(e^{-x})dx=\mathcal{F}\circ
 \chi^\ast  [f](\eta).$$
 \end{prop}

As we shall see later, both the Mellin  and Fourier  transforms extend to more general spaces which we introduce below.
\\ \\
The space of Fourier  transforms  of functions in $\mathcal{S}(\overline{\R}_+)$ plays an important role in Boutet de Monvel's pseudo-differential boundary calculus \cite{BdM}. The smoothness up to the boundary yields the Taylor expansion around $x=0$, $u(x)\sim \sum_{j\ge 0}\frac{1}{j!}u^{(j)}(0)x^{j}$. We  consider more general smooth rapidly decreasing functions allowing for  log-polyhomogeneous asymptotic behavior at zero.  To the set  $\PP=\{ (p_j,m_j)\in\C\times\N\ |\text{ for } j\in\N\}$ which prescribes the type of {polyhomogeneous singularity}  at zero, we assign the space   $\mathcal{S}_{\PP}(\R_+)$ of smooth  functions $u$  on $\R_+$ which are rapidly decreasing at infinity and such that, as $x\to 0^+$,  
$$u\sim \sum\limits_{j=0}^{\infty}\sum\limits_{k=0}^{m_j}c_{jk}x^{p_j}\ln^k x.$$
These spaces were considered by many authors, e.g. H. Kegel, B-W. Schulze, S. Rempel  \cite{RS}, \cite{RS81}, \cite{RS82};   G. Grubb \cite{G05-2}, J. Seiler,   E. Schrohe and many others.  
\\ \\
Let us also mention that Paley-Wiener type theorems for compact support  functions in the spaces $L^2(\R)$, $\mathcal{S}(\R)$ and $\mathcal{S}'(\R)$  are considered in  \cite{SB}, Theorems 11.1.1.-11.1.4.  Here  we focus in Paley-Wiener theorems for functions with positive support.
\vspace{1cm}

\textbf{Acknowledgements.} It is a pleasure to thank Sylvie Paycha, Carolina Neira  and  Elmar Schrohe for their   encouragements, suggestions and contributions to improve this paper.  I also thank Alexander Cardona for helpful comments on an earlier draft of this paper. Part of the research on which  this paper  is based  was carried out during visits to Potsdam University which the author thanks  for its hospitality. Let me also thank Sylvie Paycha and Bert-Wolfgang Schulze for the scientific advise they gave me during my stays. This research has been supported by the
\emph{Vicerrector\'ia de Investigaciones} and the \emph{Faculty of
Sciences} of the Universidad de los Andes.

\section{Functions with log-polyhomogeneous asymptotic behavior at zero}\label{S:Conormal}

In this section we study smooth rapidly decreasing functions with positive support and log-polyhomogeneous asymptotic behavior at zero.\\
We use a similar notation  to the one used in \cite{RS}. We denote by $\uP$ the set of all a sequences $\PP=\{ (p_j,m_j)\in\C\times\N\ |\   0\leq m_j\leq m \text{ for } j\in\N\}$ with
 $$\textup{Re}\, p_j\longrightarrow\infty\text{ as } j\longrightarrow \infty,\ \textup{Re}\, p_{j}\leq \textup{Re}\, p_{j+1}\ j\in\N.$$
 In particular, denote by $\PP_0$ the sequence $$\PP_0=\{ (p_j,m_j)\in\C\times\N \ |\ p_j=j,\ m_j=0,\text{for } j\in\N\},$$
corresponding to the power set associated to the usual Taylor series type expansion around $x=0$ for functions in $\mathcal{S}(\overline{\R}_+)$.

\begin{defin}\label{Def.Conormal.Sing.0}
 Let $\PP=\{(p_j,m_j)\}\in \uP$. Consider the subset $\mathcal{S}_\PP(\R_+)\subset \mathcal{S}(\R_+)$  of functions
 $u$  with the following {\em log-polyhomogeneous asymptotic expansion around zero}\index{log-polyhomogeneous asymptotic expansion around zero}
\begin{equation}\label{Conormal.Sing}
 u(x)\sim\sum\limits_{j=0}^{\infty}\sum\limits_{k=0}^{m_j}a_{jk}x^{p_j}\ln^k x,
\end{equation}
 as $ x\rightarrow 0^+.$  In other words, a  function $u$   lies in  $\mathcal{S}_P(\R_+)$ if and only if for any cut-off function $\omega$
 there exists a sequence $\{ a_{jk}=a_{jk}(u)\in\C \ |\  0\leq k\leq m_j,\text{ for } j\in\N \}$ and $N\in\N$  such that for $$u_0(x):=\omega(x)\Big( u(x)-\sum\limits_{j=0}^{N}\sum\limits_{k=0}^{m_j}a_{jk}x^{p_j}\ln^k x\Big),$$
 we have  $\partial^k_x u_0(x)=O(x^{\textup{Re}\, p_{N+1}-k})\ \forall k\in\N$, 
 and such that  $e_+(1-\omega)u\in \mathcal{S}(\R).$ In this case, we say that $u$ has a
 {\em log-polyhomogeneous asymptotic behavior at zero of type $\PP$}. This kind of singularity is also called {\em conormal singularity at zero of type $\PP$}.
\end{defin}
Similarly, denote by $\oP$ the set of all sequences $\PP=\{(p_j,m_l)\in\C\times\N\ |\ j\in\N\}$ for which $$\textup{Re}\, p_j \to -\infty\text{ as } j\longrightarrow\infty,\ \textup{Re}\, p_j \geq \textup{Re}\, p_{j+1} \ j\in\N.$$
 
\begin{defin}\label{def.sim.infty}
 Let $\PP=\{(p_j,m_j)\}\in \oP$. The subset $C^\infty_{\PP}(\R)\subset C^\infty(\R)$ denotes the set of functions
 $u$  with the following  an asymptotic expansion\index{asymptotic expansion}
\begin{equation}\label{Conormal.Sing.Infinity}
 u(x)\sim\sum\limits_{j=0}^{\infty}\sum\limits_{k=0}^{m_j}a_{jk} x^{p_j}\ln^{k} x,
\end{equation}
 as $ |x|\rightarrow \infty.$   In other words, a function $u\in C^\infty_{\PP}(\R)$ if and only if $u$ is a  smooth function on $\R$ and,
 for any cut-off function $\omega$, there exists a sequence $\{ a_{jk}=a_{jk}(u)\in\C \ |\  0\leq k\leq m_j ,\text{ for } j\in\N \}$
 and  $N\in\N$ such that
 $$u_\infty(x):=(1-\omega(x))\Big( u(x)-\sum\limits_{j=0}^{N}\sum\limits_{k=0}^{m_j }a_{jk}x^{p_j }\ln^k x\Big),$$
 then $\partial^k_x u_\infty(x)=O(x^{\textup{Re}\, p_{N+1} -k})\, \forall k\in\N.$  In this case, we say that $u$ has a
 {\em log-polyhomogeneous asymptotic behavior at infinity of type $\PP $}\index{log-polyhomogeneous asymptotic behavior at infinity}. This kind of singularity is also called
 {\em conormal singularity at infinity of type $\PP $}.
\end{defin}

For $\alpha\in\C$ and $\PP \in \uP$ (or $\PP \in \oP$), we set $T^\alpha \PP=\{(p_j+\alpha,m_j)\ |\  j\in\N\}$, which we will call the translation of $\PP$ by $\alpha$.

\begin{remark}
In Section \ref{SS:HomoD} we will consider homogeneity properties of $x^a_\pm$, and we will be advocated to restrict ourselves to {\em non-negative integer} powers of the type $x^a_\pm$ (with $a\in\C\setminus \Z_{<0}$,  cf. (\ref{Homogenity.lost})). This lead us to the following definition.
\end{remark}

\begin{defin}
 We call a power set\index{power set} $\PP=\{(p_j,m_j)\in\C\times\N \ |\ j\in\N\} \in \uP$  an {\em appropriate} power\index{appropriate power  set} set if $\PP$ contains $\PP_0=\{(j,0)\ |\ j\in\N \}$ as a subset
 and, if for all $j\in\N,$ $p_j$ is {\em not} a negative integer.
\end{defin}

\begin{example}
 Let $u(x)=x^{-1/2}e^{-x^2}$ defined for $x\in\R_+$, then $u(x) \sim \sum_{j \ge 0} {{x^{2j - {1 \over 2}} }\over {j!}}$, so that
 $u(x)$ lies in $\mathcal{S}_{T^{-\frac{1}{2}} \PP_0}(\R_+)$  with
 $T^{-\frac{1}{2}}\PP_0=\{(j-1/2,0)\ |\ j\in\N\}$.
\end{example}
Two inclusion $\PP$ and $\PP'$ between two appropriate power sets $\PP$ and $\PP'$ induces the following inclusion
$$ \mathcal{S}_{\PP'}(\R_+)\subset \mathcal{S}_{\PP}(\R_+)\;  \;  \text{ and }\; \;
 C^\infty_{\PP}(\R)\subset C^\infty_{\PP'}(\R). $$

\begin{lemma}\label{Lemma.subsets}
 For any $\PP$ an appropriate power sets we have $$\mathcal{S}(\overline{\R}_+)\subset \mathcal{S}_{\PP}(\R_+)\subset \mathcal{S}(\R_+).$$
 Moreover, the conormal singularity given by $\PP_0=\{(j,0)\ |\ j\in\N\}$ corresponds to
 the Taylor expansion of a  smooth function  up to the boundary, i.e.
\begin{equation*}
 \mathcal{S}_{\PP_0}(\R_+)=\mathcal{S}(\overline{\R}_+).
\end{equation*}
\end{lemma}

\begin{proof}
 The first assertion follows from the definition of $\mathcal{S}_{\PP}(\R_+)$ and the fact that $\PP$ is an {\em appropriate} power sets. Let us prove the second assertion. First,
 $\mathcal{S}_{\PP_0}(\R_+)\supset\mathcal{S}(\overline{\R}_+)$  follows  from Taylor's asymptotic expansion at $x=0^+$. To see
 $\mathcal{S}_{\PP_0}(\R_+)\subset\mathcal{S}(\overline{\R}_+)$,
 let $u\in S_{\PP_0}(\R_+)$. This assumption implies that $u$ is a rapidly decreasing
 function and smooth for $x \gg 0$. The existence of $\tilde{u}\in S(\R)$ such that
 $\tilde{u}|_{\R_+}=u$ then directly follows from the main result in
 \cite{Seeley1973} where a smooth function defined in a half space, all
 of whose derivatives have continuous limits at the boundary, is extended
 to a $C^\infty$-function in the whole space.
\end{proof}

Finally, let us remark that in \cite{RS}, Section 1.2, the authors show that $\mathcal{S}_\PP(\R_+)$ may be equipped with a Fr\'echet topology derived from the $L^2$-scalar product, however, such topology is outside from the aims of this document.
\section{Mellin transform and Paley-Wiener theorems}\label{S:MellinPaleyWiener}

Following \cite{FGD} and \cite{P79}, we {now recall } Paley-Wiener {type} theorems for the Mellin transform. We begin with a brief summary of the main properties of the Mellin transform we will need.

\subsection{The Mellin transform and some of its properties}\label{SS:MellinProperties}

{\bf Basic properties of the Mellin transform.} We follow \cite{FGD} to recall some properties of the Mellin transform we will use throughout this Chapter.\\
\\
\textit{Fundamental strip:} Let $f(t)$ be a continuous function on $\R_+$ such that
$$f(t)=\begin{cases}  O(t^{-\alpha}) &\mbox{ for } t\rightarrow 0,\\
  O(t^{-\beta}) &\mbox{ for } t\rightarrow \infty.
 \end{cases}$$
for some real numbers $\alpha<\beta$. The convergence of $\mathcal{M}[f]$ follows from the inequality
$$|\mathcal{M}[f](s)|\leq C\int_0^1 t^{\textup{Re}\, s-1+\alpha}dt+c\int_1^\infty t^{\textup{Re}\, s-1+\beta}dt,$$
where $c$ and $C$ denote some constants. Then, $\mathcal{M}[f](s)$ exists for any $s\in\C$ in the strip $\{\alpha<\textup{Re}\, s<\beta \}$. The largest open strip in which the integral (\ref{Mellin.integral}) converges
is called the {\em fundamental strip} for $f$. By the Cauchy-Riemann equations, $\mathcal{M}[f](s)$ is analytic inside its fundamental strip.\\
\\
By means of the change of variables $t=e^{-x}$, and the formula for inverse of the Fourier transform (\ref{Eq:def.Fourier}), we can obtain a formula for the inverse of the Mellin transform.

\begin{prop}\label{Inverse.Mellin.formula} {\em( \cite{FGD}, Theorem 2.)}
 Let $f(t)$ be a continuous integrable function  with fundamental strip $\la a,b\ra$. If there exists a real number $c\in (a,b)$ such that $\mathcal{M}[f](c + it)$ is integrable. Then, for $s=c+i\eta $, the following equality holds
 $$f(t) =\frac{1}{2\pi i}\int_{c-i\infty}^{c+i\infty}t^{-(c+i\eta)} \mathcal{M}[f](c+i\eta)d\eta.$$
\end{prop}

The Mellin and Fourier transform have similar properties due to relation (\ref{Fourier.Mellin.Rel}). Some of these properties are enunciated in the following Proposition. A proof of these properties follows by straight forward calculations (see \cite{FGD}).

\begin{prop}\label{Properties.Mellin}
 Let $f$ be a function whose Mellin transform admits  the fundamental strip $\la a,b\ra$\index{fundamental strip}, and let $p$ be a non zero real number, $r,q$ positive real numbers. We have:

%
\begin{enumerate}
\item [\textup{(i)}] $ \mathcal{M}[t^zf(t)](s)=\mathcal{M}[f](s+z),  \text{ on }\la a-\textup{Re}\, z,b-\textup{Re}\, z\ra$;
\item [\textup{(ii)}] $  \mathcal{M}[f(t^p)](s)=\frac{1}{p}\mathcal{M}[f]   \big(    \frac{s}{p}   \big),  \text{ on }\la pa,pb\ra$;
\item [\textup{(iii)}] $ \mathcal{M} \big[    \frac{d}{dt}f(t)  \big] (s)= -(s - 1)\mathcal{M}[f](s - 1),  \text{ on }\la a-1,b-1 \ra$;
\item [\textup{(iv)}] $ \mathcal{M} \big[   \ln t f(t) \big]  (s)=\frac{d}{ds} \mathcal{M}[f](s),  \text{ on }\la a,b\ra$;
\item [\textup{(v)}] $ \mathcal{M}[t\frac{d}{dt}f(t)](s)=-s\mathcal{M}[f](s)$;
\item [\textup{(vi)}]$ \mathcal{M}\Big[\int_0^t f(x)dx\Big](s)=-\frac{1}{s}\mathcal{M}[f(t)](s)$.
\end{enumerate}
 Finally, we denote by $H$ the Heaviside function on $[0,1]$ then
\begin{enumerate}
\item [\textup{(vii)}]
$\mathcal{M}[H(t)](s)=\frac{1}{s}\ \text{( resp. } \mathcal{M}[1-H(t)]=-\frac{1}{s}),  \text{ on }\la 0,\infty\ra \text{ ( resp. } \text{ on }\la -\infty,0\ra)$;
\item [\textup{(viii)}]
$\mathcal{M}[H(t)t^d\ln^k t](s)=\frac{(-1)^kk!}{(s+d)^{k+1}}   \text{ on }\la -d,\infty\ra$;
\item [\textup{(ix)}]
$\mathcal{M}[(1-H(t))t^d\ln^k t](s)=-\frac{(-1)^kk!}{(s+d)^{k+1}}   \text{ on }\la -\infty, d\ra$.
\end{enumerate}

\end{prop}

\subsection{A Paley-Wiener theorem for Schwartz functions smooth up to the boundary with positive support}\label{P-W.mellin.smooth}

Let $\beta\in\C$, in the following $\Gamma_{\beta}$ will denote the vertical complex line $\{s\in\C\ |\  \textup{Re}\, s =\beta \}$.

\begin{prop}\label{Prop.compac-mellin}{\em (\cite{P79}, Theorem 3.1.)}
A complex function $h$  on $\C$. Then $h$ is a Mellin transform of some $u\in C^\infty_c(\R_+)$
 if and only if $h$ is analytic, and for any positive integer $m$ there exist constants  $C_m$ and $a>0$ such that
\begin{equation*}
|h(s)|\leq C_m \la  s\ra^{-m}e^{a|\textup{Re}\, s |}.
\end{equation*}
 In this case, we have $\textup{supp}\ u\subset [e^{-a},e^a]$.
\end{prop}

\begin{lemma}\label{Lemma.Planchelet}
 The Mellin transform $\mathcal{M}$ maps $C^\infty_c(\R_+)$ to $L^2(\Gamma_{\frac{1}{2}})$ and extends to $$\mathcal{M}:L^2(\R_+)\rightarrow L^2(\Gamma_{\frac{1}{2}}),$$
 which is an isometric isomorphism, where the norm on $L^2(\Gamma_{\frac{1}{2}})$ is given by
 $$\|\mathcal{M}[f]\|_{L^2(\Gamma_{\frac{1}{2}})}=\int_{-\infty}^\infty
 \Big|\mathcal{M}[f]\Big(\frac{1}{2}+i\eta\Big)\Big|^2d\eta.$$
\end{lemma}


We recall a a Paley-Wiener type theorem.

\begin{prop}\label{Prop.L2}{\em (\cite{P79}, Theorem 3.2.)}
 Let $u\in L^2(\R_+)$, and assume $\textup{supp}\,  u\subset (0,e^a]$ for some $a>0$. Then
 the Mellin transform $h(s):=\mathcal{M}[u](s)$ extends from $\Gamma_{1/2}$ to an analytic function
 $h(s)$ in $\{\textup{Re}\, s>1/2\}$ such that
 \begin{itemize}
\item[(i)]  $h_\xi(\eta):=h(\xi+1/2+i\eta)\in L^2(\R_\eta)$ for every $\xi>0,$ and satisfies
 $$h_\xi(\eta)\longrightarrow \mathcal{M}[u](1/2+i\eta)\text{ in } L^2(\R_\eta)\text{ as }\xi\rightarrow 0^+;$$
\item[(ii)] there exists a constant $C>0$ such  that
\begin{equation}\label{Nadia.1.1.37}
\| h_\xi\|_{L^2(\R_\eta)}\leq Ce^{a\xi}.
\end{equation}
\end{itemize}
 Conversely, if $h$ is an analytic function in $\{\textup{Re}\, s>1/2\}$ satisfying the estimates (\ref{Nadia.1.1.37})
 for some constants $C,a>0$, then $h$ is the Mellin transform of a function $u\in L^2(\R_+)$ with $\textup{supp}\, u\subset (0,e^a]$.
\end{prop}

We may restrict the domain of the map $\mathcal{M}:L^2(\R_+)\longrightarrow L^2(\Gamma_{1/2})$ to $\mathcal{S}(\overline{\R}_+)$. 
Proposition \ref{Smooth.P-W.Mellin} can be compared with  Theorem \ref{Paley-Wiener.thm.one.sided}, concerning the Fourier transform. The following result which we refer to as the Paley-Wiener theorem for functions smooth  {up} to the boundary,  {yields} a singular expansion:
 \\
\\
Let us recall a useful definition (see, e.g. \cite{FGD}): We call a {\em singular expansion} of a meromorphic function $h(s)$ in $\Omega\subset\C$, denoted by $\cong$, a formal sum of singular  elements of $h(s)$ at each pole of $h$ in $\Omega$. It is basically a sum of Laurent expansions around all poles truncated
 to the $O(1)$. For instance,
\begin{equation}\label{Eq:exam.sing.expans}
\frac{1}{s^2(s+1)}\cong \Big[\frac{1}{s+1}+2\Big]_{s=-1} +\Big[\frac{1}{s^2}-\frac{1}{s}\Big]_{s=0}+\Big[\frac{1}{2}\Big]_{s=1}\text{ for } s\in(-2,2).
\end{equation}

\begin{prop}\label{Smooth.P-W.Mellin}
The Mellin transform $\mathcal{M}$ maps functions $u$ in $\mathcal{S}(\overline{\R}_+)$ to holomorphic unctions $\mathcal{M}[u]$ on the strip $\la 0,\infty\ra$ which can be extended to meromorphic function on $\C$ with simple poles at $s=-j$ for $j\in\N,$ i.e.
 $\mathcal{M}[u](s)\cong \sum\limits_{j=0}^\infty a_j \frac{1}{(s+j)}$,
 and  such that for any excision function $\chi$ for the set $\{-j\}_{j\in\N}$, we have $$\chi(s)\mathcal{M}[u](s)\big|_{\Gamma_\eta}\in\mathcal{S}(\Gamma_\eta).$$ 
\end{prop}

A proof can be done by using the Taylor expansion of $u(x)$ around zero; Proposition \ref{Prop.compac-mellin} and some properties of the Mellin transform in Proposition \ref{Properties.Mellin}. However, we will give a proof of this fact in Section \ref{SS:P-W.Mellin.log-poly}, using results involving Theorem \ref{P-W.Mellin.Asymp.Log-poly}, below.

\begin{remark}
Notice that if we set $\tilde{\mathcal{M}}[u](s):=\frac{1}{\Gamma(s)}\mathcal{M}[u](s)$, for $u\in\mathcal{S}(\overline{\R}_+)$, where $\Gamma(s):=\int_0^\infty t^{s-1}e^{-t}dt$ is the Gamma function,  integration by parts shows $\Gamma$ has simple poles at $-j,\ j\in\N$. {It therefore follows} from Proposition \ref{Smooth.P-W.Mellin} {that} $\tilde{\mathcal{M}}$ maps functions in $\mathcal{S}(\overline{\R}_+)$
 to  {holomorphic} functions in $\C$ .  We call $\tilde{\mathcal{M}}[u](s)$ the {\em normalized Mellin transform}.
\end{remark}


\subsection{A Paley-Wiener theorem for functions with log- polyhomogeneous asymptotic behavior at zero and positive
support}\label{SS:P-W.Mellin.log-poly}

In this section, using known Paley-Wiener type theorems,  which are recalling here,  we characterize the image under the Mellin transform of spaces $\mathcal{S}_{\PP}(\R_+)$ (cf. Theorem \ref{P-W.Mellin.Asymp.Log-poly}). \\
\\
We first prove a result relative the Mellin transform of functions of the type $\omega(t)t^d\ln^k t$, for $\omega$ a cut-off function.

\begin{lemma}\label{Lemma.Mellin.cut-off}
{For any a cut-off function $\omega$, the Mellin transform of}  $$\omega_{dk}(t):=\omega(t)t^d\ln^k t \;  \text{ and }\;  \tilde{\omega}_{dk}(t):=(1-\omega)(t)t^d\ln^k t$$    are meromorphic functions on $\C$  and $$\mathcal{M}[\omega_{dk}](s)=-\mathcal{M}[\tilde{\omega}_{dk}](s)=\frac{(-1)^kk!}{(s+d)^{k+1}}+h_{dk}(s),$$
 where $h_{dk}$ is an {entire}  function which is rapidly decreasing  in $t$ along parallel lines to $c+it$ with $c\in\la -d,\infty\ra$.
\end{lemma}

\begin{proof}
 We prove the result for $\omega_{dk}$, for $\tilde{\omega}_{dk}$ it is enough to observe that $\mathcal{M}[f(\frac{1}{t})](s)=-\mathcal{M}[f(t)](s)$. Let $\chi(x)=e^{-x}\  x\in\R$. For $f(x):=(\chi^\ast \omega_{dk})(x)$ and $\omega'_{dk}(x):=(\chi^\ast\omega_{dk})(x),$ we have $$f(x)=\omega'_{dk}(x)x^{-k}e^{-dx}.$$ We write $\omega'_{dk}(x)=H(x)+(\omega'_{dk}(x)-H(x))$, since the difference $\omega'_{dk}-H$  is a compactly supported function we have $h_{dk}(s)=\mathcal{M}[\omega'_{dk}-H](s)$ is an analytic function on $\C$. Applying (\ref{Fourier.Mellin.Rel}) to $f$ we have that 
 $$\mathcal{M}[\omega'_{dk}](\xi+i\eta)=\mathcal{F}^{-1}[H(x)e^{-x\xi}f(x)](\eta)+h_{dk}(s).$$
 Finally, setting $s=\xi+i\eta$ and using that $\mathcal{F}[x^nf(x)]=i^n\partial^n_\xi \mathcal{F}[f](\xi)$ and $\mathcal{F}[H(x)e^{-ax}]=(a+i\xi)^{-1}$, we obtain
\begin{eqnarray*}
\mathcal{M}[\omega'_{dk}](s) &=&\int_{-\infty}^{\infty}e^{ix(\eta+i\xi)}H(x)x^{-k}e^{-xd}\normd  x+h_{dk}(s)\\
 &=& \partial^k_{\xi+i\eta}[(\xi+i\eta)+d )^{-1}]+h_{dk}(s)\\
 &=&\frac{(-1)^kk!}{(s+d)^{k+1}}+h_{dk}(s).
\end{eqnarray*}
\end{proof}
%
The following Theorem gives us a singular expansion for the Mellin transform of functions with log-polyhomogeneous behaviour around zero and infinity. We now recall two important results in \cite{FGD}, called there {\em Direct Mapping Theorem} and {\em Converse Mapping Theorem}.

\begin{prop}\label{Direct.Thm}{\rm [Direct Mapping Theorem]}
 Let $f(t)$ be a continuous function with Mellin transform $\mathcal{M}[f](s)$ defined in non-empty strip $\la a, b\ra$.
\begin{enumerate}
\item Assume that $f(t)$ admits the following asymptotic expansion for $t\rightarrow 0$:
 $$f(t)=\sum\limits_{j=0}^{M-1}\sum\limits_{k=0}^{m_j}a_{jk} t^{p_j}\ln^{k} t+O(t^M),$$ where
 $-M<-\textup{Re}\, p_j\leq a$ with $M\in\Z$. Then $\mathcal{M}[f](s)$ has a meromorphic continuation to the strip $\la -M,  b\ra$, and  for $ s\in\la -M,b\ra$:
 $$\mathcal{M}[f](s)\cong \sum\limits_{j=0}^\infty\sum\limits_{k=0}^{m_j}a_{jk}\frac{(-1)^{k}k!}{(s+p_j)^{k+1}}.$$

\item Let $f(t)$ have the following asymptotic expansion for $t\rightarrow\infty$:
 $$f(t)=\sum\limits_{j=0}^{M-1}\sum\limits_{k=0}^{m_j}a_{jk} t^{-p_j}\ln^{k} t+O(t^{-M}),$$
 where $b\leq -\textup{Re}\, p_j<M$ with $M\in\Z$. Then $\mathcal{M}[f](s)$ has a meromorphic continuation to the strip $\la a, M \ra$, and for $ s\in\la a,M \ra$:
 $$\mathcal{M}[f](s)\cong-\sum\limits_{j=0}^\infty\sum\limits_{k=0}^{m_j}a_{jk}\frac{(-1)^{k}k!}{(s+p_j)^{k+1}}.$$
\end{enumerate}
\end{prop}


\begin{prop}\label{Reverse.Thm}{\rm [Converse Mapping Theorem]}
%
%
 Let $f(t)$ be continuous on $(0,\infty)$ with Mellin transform
 $\mathcal{M}[f](s)$ with non-empty fundamental strip $\la a, b\ra$.
\begin{enumerate}
\item
 \begin{itemize}
 \item Assume that $h(s)$ admits a meromorphic continuation in the strip $\la -M,b\ra$ for some $-M<a$ with a finite number  of poles in $\la-M,a \ra$, and is analytic on $\textup{Re}\, s=-M$.
  Moreover we assume that it admits the following singular expansion:
  \begin{equation}\label{Singular.Expansion}
  \mathcal{M}[f](s)\cong \sum\limits_{j=0}^\infty\sum\limits_{k=0}^{m_j}a_{jk}\frac{(-1)^{k}k!}{(s+p_j)^{k+1}},\text{ for }s\in \la -M,a\ra,
  \end{equation}
 \item  and that there exists a real number $c$ in $\la   a,b\ra$  satisfying the estimate
 \begin{equation}\label{Integral.Estimate}
\exists r>1,\quad \mathcal{M}[f](s)=O(|s|^{-r}),
 \end{equation}
  when $|s|\rightarrow\infty \text{ in } -M\leq \textup{Re}\, s\leq c$.
 \end{itemize}
 Then $f(t)=\frac{1}{2\pi i}\int_{c-i\infty}^{c+i\infty}t^{-s}h(s)ds$ has the following asymptotic expansion at $0$,
 $$f(t)=\sum\limits_{j=0}^{M-1}\sum\limits_{k=0}^{m_j}  a_{jk} t^{p_j}\ln^{k} t+O(t^{M}).$$

\item
 \begin{itemize}
 \item Assume that $h(s)$ admits a meromorphic continuation to the strip $\la a,M\ra$ for some $M>b$
 with a finite numbers of poles $\la  a,M\ra$, and is analytic on $\textup{Re}\, s=M$.
 Moreover, {we assume that?} it admits the singular expansion (\ref{Singular.Expansion}) for $s\in \la c,M\ra.$
 \item We also assume the existence of real number $c$ in $\la  a,b \ra$ such that such that for some $r>1$,
 satisfies the estimate  (\ref{Integral.Estimate}) when $|s|\rightarrow\infty \text{ in } c\leq \textup{Re}\, s\leq M.$
 \end{itemize}
 then $f(t):=\frac{1}{2\pi i}\int_{c-i\infty}^{c+i\infty}t^{-s}h(s)ds$ and it has the following asymptotic expansion at $\infty$,
 $$f(t)=-\sum\limits_{j=0}^{M-1}\sum\limits_{k=0}^{m_j}a_{jk}t^{p_j}\ln^{k} t+O(t^{-M}).$$
\end{enumerate}
\end{prop}

{Let us now}  characterize the image under the Mellin transform of the spaces $\mathcal{S}_{\PP}(\R_+)$ (see, \ref{Def.Conormal.Sing.0}).

\begin{thm}\label{P-W.Mellin.Asymp.Log-poly}
Let $a\in\C$. An analytic function $h(s)$ on the strip $\la a,\infty\ra$ is a Mellin transform of a function $f$ in $\mathcal{S}_{\PP}(\R_+)$ with asymptotic expansion for $t\rightarrow 0$ of the form
\begin{equation}\label{Eq:asymp-function}
 f(t)=\sum\limits_{j=0}^{M-1}\sum\limits_{k=0}^{m_j}a_{jk} t^{p_j}\ln^{k} t+O(t^M),
\end{equation}
where  $-M<-\textup{Re}\, p_j\leq a,$   if and only if

\begin{itemize}
\item[(i)]
$h(s)$ is defined for $\textup{Re}\, s>-\textup{min}_{j\in\N}\, \textup{Re}\, p_j$ {and} has  {a} meromorphic continuation in the strip $\la  -M,\infty\ra$ with singular expansion
 \begin{equation}\label{Eq.Cor.}
 h(s)\cong \sum\limits_{j=0}^{M-1}\sum\limits_{k=0}^{m_j}a_{jk}\frac{(-1)^{k}k!}{(s+p_j)^{k+1}}
 \end{equation}
 for $s\in \la M,a\ra$, and it is analytic on $\textup{Re}\,(s)=-M$,

\item[(ii)] there exists a real number $c\in\la a,\infty\ra$  such that for some $r>1$,
$$h(s)=O(|s|^{-r}),$$ when $|s|\rightarrow \infty$ in $-M\leq \textup{Re}\, s\leq c.$
\end{itemize}

\end{thm}

\begin{proof}

Let $\PP=\{(p_j,m_j)\}_{j\in\N}$. Let $f\in\mathcal{S}_{\PP}(\R_+)$ with asymptotic expansion as before and a cut-off function $\omega$. Then
 $$f(t)=\omega(t)\Big( \sum\limits_{j=0}^{M-1}\sum\limits_{k=0}^{m_j}a_{jk} t^{p_j}\ln^{k} t+O(t^M)   \Big)+
 (1-\omega(t))f(t).$$ 
Notice that the  {result} is independent of the choice of the cut-off function $\omega$ from the fact that $\omega-\omega'$ has compact support for any other cut-off function $\omega'$, and the Mellin transform of a function with compact support is an analytic function.  We know from Proposition \ref{Smooth.P-W.Mellin} that the Mellin transform of $ (1-\omega(t))f(t)$ is an analytic function defined in the whole $\C$. From Proposition \ref{Direct.Thm}, $h(s)=\mathcal{M}[f](s)$  admits a meromorphic continuation in the strip $\la  -M,\infty\ra$ with singular expansion (\ref{Eq.Cor.}). The estimate $h(s)=O(|s|^{-r})$ comes from Proposition \ref{Reverse.Thm}.
\\
 Let us now prove {the} converse statement. Let $h$ an meromorphic function satisfying {the conditions} of Theorem \ref{P-W.Mellin.Asymp.Log-poly}. Let $c\in\la a,b\ra$, and set
 $$f(t)=\int_{c-i\infty}^{c+i\infty}t^{-s}h(s)ds,$$
 it follows from Theorem \ref{Reverse.Thm} that $f(t)$ satisfies (\ref{Eq:asymp-function}). 
It now remains to show that $f$ is a  rapidly decreasing  smooth function.
 The smooth condition follows from Proposition \ref{Mellin.Dist.} and the equality
$$\frac{d^k}{dt^k}f(t)=\frac{1}{2\pi i}\int_{c-i\infty}^{c+i\infty}(s-1)\cdots (s-k)t^{-s}h(s)ds.$$
The rapidly decreasing condition follows from the fact that, for every $l,k\in\N$, $$|t^l\frac{d^k}{dt^k}f(t)|\leq C_m\int_{c-i\infty}^{c+i\infty}|t^{-s+l}\la s \ra^{k}h(s)|ds \leq C_m\int_{c-i\infty}^{c+i\infty}t^{l}\la s\ra^{k-m}ds<\infty,$$
for $m$ chosen large enough.
\end{proof}

We finally come to the proof of the Proposition \ref{Smooth.P-W.Mellin}.

\begin{proof}[Proof of Proposition \ref{Smooth.P-W.Mellin}.]
Let $u\in\mathcal{S}(\overline{\R}_+)$. We have that for any integer $M>0$, the following asymptotic expansion
$$u(t)=\sum\limits_{j=0}^{M-1}a_jt^{j}+O(t^M).$$
It follows from Theorem \ref{P-W.Mellin.Asymp.Log-poly}, $\mathcal{M}[u](s)$ is an analytic function on $\textup{Re}\, s>0$ and has a meromorphic continuation in $\C$ with singular expansion   
\begin{equation}
h(s)\cong \sum\limits_{j=0}^{M-1}a_{j}\frac{1}{s+p_j}
\end{equation}
for $s\in \la-\infty,0\ra$.\\
Now, for $s=\xi+i\eta$ and fixed $\xi$, for any $\alpha, \beta\in \N$ and any excision function $\chi$ for the set $\{j\}_{j\in\N}$, since (cf. \ref{Fourier.Mellin.Rel}),  there exists a positive constant $C_{\alpha,\beta}$ such that 
\begin{eqnarray*}
\sup_{\eta\in \Gamma_{\eta}}|\eta^\alpha \partial^\beta_{\eta}\chi\mathcal{M}[u](s)|=\sup_{\eta\in \Gamma_{\eta}}|\eta^\alpha \partial^\beta_{\eta}\chi\mathcal{F}[e^{-\xi t}u(e^{-t})](\eta)|\leq C_{\alpha,\beta}, 
\end{eqnarray*}
the last inequality follows from well-known properties of the Fourier transform and the fact of $u$ lies in $\mathcal{S}(\overline{\R}_+)$. Therefore, $\chi(s)\mathcal{M}[u](s)\in\mathcal{S}(\Gamma_{\eta})$.
\end{proof}

\subsection{A Paley-Wiener theorem for extendable distributions with positive support}\label{Subsec-extend-distr}

In this section we follow \cite{P79} in order to recall the main properties of the Mellin transform acting on distributions we are interested in.
\\ \\
{\bf The space $\mathcal{E}_+^\prime(\R_+)$.} Let $\mathcal{E}'_+(\R_+)$ denote the subset of extendable distributions in $\mathcal{D}'(\R_+)$ with compact support in $\R$, i.e. $U\in \mathcal{D}'_+(\R_+)$ if and only if there exists $V\in\mathcal{D}'(\R)$ with compact support such that $r_+V=U$, here $r_+$ denotes the restriction map to $\R_+$.  The space $\mathcal{E}'_+(\R_+)$ may be equipped with the following algebraic structure: Let $U_i,\ i=1,2,$ be distribution in  $\mathcal{E}'(\R_+)$. The convolution product $U_1\ast U_2$ of $U_1$ and $U_2$ given by
$$\la U_1\ast U_2,\phi\ra:=\la U_1(s)\otimes U_2(t),\phi(st)\ra,\ \forall \phi\in\mathcal{D}'(\R_+).$$

Let $\Omega$ be an open subset of $\R$.  The structure theorem  for distributions shows that for any $U\in\mathcal{D}'(\Omega)$ and any compact set $K\subset \Omega$, there exists
 a continuous function $F\in C^0(K)$ and a multi index $\alpha$, such that $U=D^\alpha F$ on $K$. In this case,  $\alpha$ is called the order of $U$ (see, e.g. \cite{H}).
 
\begin{prop}{\em (\cite{Schwartz68}, Theorem 24)}
 For  {any distribution $V$ with compact support
 in $\R$ which extends some $U\in \mathcal{E}_+^\prime(\R_+)$,} there exists $m\in\N$, such that $V$ is a distribution of order $\leq m$ and  the function $$s\mapsto \mathcal{M}[U](s)=\la V,r_+ t^{s-1} \ra$$
 is well defined as holomorphic function for $\textup{Re} \, (s)>m+1$ and independent of the choice of {extension} $V$ of $U$. The function  $\mathcal{M}[U](s)$ is called the Mellin transform of $U$.
\end{prop}

{\bf The space $\mathcal{H}'_+$.}  {Let  $\mathcal{H}'_+$ be} the space of holomorphic functions $h$ defined on  {half}-planes $\textup{Re}\, s>r \ (r\in\R)$, identifying two {such functions if they coincide in some half-}plane, and satisfy
\begin{equation}\label{Bounded.Mellin}
 |h(s)|\leq C\la s\ra^m e^{a\, \textup{Re}\, s},\text{ for } \textup{Re}\, s >r,
\end{equation}
 for some constants $C>0,\ m\in\Z,a>0$ and $r\in\R$, which can depend on $h$.
 The space $\mathcal{H}'_+$ may be equipped with the following algebraic structure:
 Let $h_i,\ i=1,2,$ be functions in $\mathcal{H}'_+$. The product $h_1h_2$ of $h_1$ and $h_2$ is given by $$s\mapsto h_1(s)h_2(s).$$

\begin{lemma}
 The Mellin transform $\mathcal{M}[U](s)$ of  any $U\in\mathcal{E}'_+(\R_+)$ lies in $\mathcal{H}'_+$.
\end{lemma}
\begin{proof}
 For any $U\in\mathcal{E}_+^\prime(\R_+)$ with support in $(0,a)$, there exists $f\in C^0(\R_+)$ and $m\in\N$
 such that $\textup{supp}\,  f\subset (0,a)$ and $U=D^mf.$ Then, applying properties of the Mellin transform we obtain
 $$ \mathcal{M}[U](s)=(-1)^m (s-1)(s-2)\cdots (s-m)\int_0^a t^{s-m-1}f(t)dt,$$
 and therefore $\mathcal{M}[U](s)$ therefore satisfies the estimate (\ref{Bounded.Mellin}) for $r>m.$
\end{proof}

 For $h\in\mathcal{H}'_+$ satisfying the estimate (\ref{Bounded.Mellin}) and $c>r$, we define
 $$\mathcal{R}[h](t):=\dlim_{A\rightarrow \infty} \frac{1}{2\pi i}\int_{c-iA}^{c+iA}t^{-s}h(s)ds.$$
 As it is described in the next Theorem, the map $\mathcal{R}$ defines
 a linear map from $\mathcal{H}'_+$ to $\mathcal{E}'_+(\R_+)$,
 and the distribution $\mathcal{R}[h](t)$ is called \textit{the inverse Mellin transform} of $h$.
 
\begin{prop}{\em(\cite{P79}, Theorem 2.3.)}
 The Mellin transform $\mathcal{M}$ is an algebra isomorphism from $\mathcal{E}'_+(\R_+)$ to $\mathcal{H}'_+$,
 with inverse $\mathcal{R}$. Moreover, if $U\in\mathcal{E}'_+(\R_+)$ and $h(s)=\mathcal{M}[U](s)\in \mathcal{H}'_+$, the support
 of $U$ is contained in $(0,c]$ if and only if $h$ verifies $(\ref{Bounded.Mellin})$ with $a=c$.
\end{prop}

{For any  $h\in \mathcal{H}'_+$, the distribution $\mathcal{R}[h](t)$ is} called \textit{the inverse Mellin transform}\index{inverse Mellin transform} of $h$.
{The next Theorem which concerns Mellin transforms} of extendable compactly supported distributions, { will be later stated in the language of Fourier transforms of tempered distributions} with positive support in section \ref{SS:WP-thm.tempred.distributions}, Proposition \ref{WP-thm.tempred.distributions}.

\begin{prop}\label{Mellin.Dist.}{\em(\cite{P79}, Theorem.3.1.)}
The function $h\in \mathcal{H}'_+$ is the Mellin transform $h(s)=\mathcal{M}[U](s)$ of a distribution $U$ in $\mathcal{E}'_+(\R_+)$ with support $[e^{-a},e^a]$, $a>0$ if and only if
$h(s)$ is an {entire} function and  the following estimate {holds} for some $ m\in\N,$ $$|h(s)|\leq \la s\ra^m e^{a |\textup{Re}\, s|},\ s\in\C.$$
Moreover,  $h\in \mathcal{H}'_+$ is the Mellin transform $h(s)=\mathcal{M}[U](s)$ of a $C^\infty$-function $U$ with support $[e^{-a},e^a],\ a>0$ if and only if
$h$ is an {entire} function and for all $ m\in\N$ there exits $C_m>0$ such that $$|h(s)|\leq C_m\la s\ra^{-m} e^{a|\textup{Re}\, s|},\ s\in\C.$$
\end{prop}


\section{Fourier transform and Paley-Wiener theorems}

In this section we discuss Paley-Wiener type results for the Fourier transform, similar to  those  {derived in Section} \ref{S:MellinPaleyWiener} for the Mellin transform. We relate  the Fourier transform of functions with positive/ negative support and  log-polyhomogeneous asymptotic expansion at zero with analytical functions in $\mathcal{A}(\C_\mp)$ having an asymptotic expansion at infinity (cf. Theorem \ref{Type Paley-Wiener.thm's}).

\subsection{Homogeneous and log-homogeneous distributions and their the Fourier transform }\label{SS:HomoD}

In this section we summarize known results on homogeneous and  log-homogeneous distributions  (see also \cite{Eskin}, \cite{H}, section 3; \cite{GO}; \cite{FG}) and we comment along the way, additionally, we compute explicitly the Fourier (and the inverse) transform of the mentioned distributions, see Proposition \ref{Fourier.transform.homogeneous}, \ref{Inverse.Fourier.transform.homogeneous} and Corollary \ref{Fourier.transform.log-homogeneous}.\\

For $x \in \R$ and $a\in\C$, with $\textup{Re}\, a>-1$, let $x^a_\pm$ be a \textit{homogeneous tempered distributions} defined by the local integrable functions
\begin{equation*}
  x^a_+:= \begin{cases} x^a &\mbox{if } x>0 \\
  0 & \mbox{if } x\leq 0, \end{cases}
\;\; \text{ and } \;\;
  x^a_-:= \begin{cases} |x|^a &\mbox{if } x<0 \\
  0 & \mbox{if } x\geq 0. \end{cases}
\end{equation*}
For any $x \in \R$, we have
\begin{equation}\label{3.2.1.HI}
 x x_\pm^a=x_\pm^{a+1},\text{ if } \ \textup{Re}\, a >-1,
\end{equation}
 and, using the distributional extension of a derivative (Theorem  {\rm 3.1.3.}  in \cite{H}), we have
\begin{equation}\label{3.2.2.HI}
 \frac{d}{d x} x_\pm^a=a x_\pm^{a-1},\text{ if } \ \textup{Re}\, a>0.
\end{equation}
If $a=0$ then  the l.h.s. of (\ref{3.2.2.HI}) is the Heaviside function $H^\prime(x)=\delta_0$ and  the r.h.s. is  zero.\\

Two methods are presented in \cite{H} to extend   $x^a_\pm$  to all $a\in \C$  as a tempered distribution,  preserving when possible properties  (\ref{3.2.1.HI}) and (\ref{3.2.2.HI}). We illustrate one of them, for $x^a_+$, called Riesz's Method\index{Riesz's Method}.
 For fixed $\phi\in C^\infty_c(\R)$,  set
\begin{equation*}
 a\mapsto I_a(\phi):=\la  x_+^a,\phi \ra=\int_0^\infty x^a\phi(x)dx.
\end{equation*}
Integration by parts $k$ times yields
\begin{equation}\label{Eq.3.Hormander}
 I_a(\phi)=\frac{(-1)^k}{(a+1)\cdots (a+k)}I_{a+k}(\phi^{(k)}).
\end{equation}
The right-hand side is analytic for $\textup{Re}\, a>-k-1$ outside a set of simple poles $\{-1,-2,-3\cdots, -k\}$. Then, for $a\notin\Z_{<0},\ \textup{Re}\, a+k>-1$ we can define $x^a_\pm$ as the tempered distributions 
\begin{equation}\label{Eq:Homog.Distrib}
x^a_+=\frac{\partial^k_x x^{a+k}_+ }{(a+1)\cdots (a+k)}\;\;\; \text{ and }\;\;\;
 x^a_-=\frac{(-1)^k \partial^k_x x^{a+k}_- }{(a+1)\cdots (a+k)}.
\end{equation} 
 Moreover, the residue of $a\mapsto I_a(\phi)$ at $a=-k$, for $k\in\Z_{>0}$ is given by
 $$\lim_{a\rightarrow -k}(a+k)I_{a}(\phi)=\lim_{a\rightarrow -k}\frac{ (-1)^k I_a(\phi) }{ (a+1)\cdots (a+k-1) }=\phi^{(k-1)}(0)/(k-1)!,$$
 i.e. $\Res_{a=-k} x^{a}_\pm=\frac{(\mp 1)^{k-1}}{(k-1)!}\delta^{(k-1)},$  so subtracting the singular part, we obtain for $x^a_+$
 $$x^{-k}_+(\phi):=\lim_{\epsilon\rightarrow 0} I_a(\phi)-\frac{\phi^{(k-1)}(0)}{\epsilon(k-1)!}=-\int_0^\infty \frac{\ln x \phi^{(k)}}{(k-1)!}dx+\Big(\sum\limits_{j=1}^{k-1} \frac{1}{j}\Big)
 \frac{\phi^{(k-1)}(0)}{(k-1)!}.$$
\begin{remark}
 This extension is unique as a result of the uniqueness of the analytic continuation.
\end{remark}

\vspace{0.2cm}
 For $a\in\C\setminus\Z_{\leq  0}$, we have $$\frac{d}{dx}x^{a}=ax^{a-1}_+ \;\;  \text{ and }\;\;
 \la x^a_+,\phi\ra=t^a\la x^a_+,\phi_t\ra,$$ where $\phi_t(x)=\phi(tx).$ However, for $k\in\Z_{<0},$ we have
 $$\frac{d}{dx}x^{-k}_+=-kx^{-k-1}_+  + \frac{(-1)^k}{k!}\delta^{(k)}_0,\text{ and }$$
\begin{equation}\label{Homogenity.lost}
 \la x^{-k}_+,\phi\ra=t^{-k}\la x^{-k}_+, \phi\ra +\frac{(-1)^{k-1}}{(k-1)!} \la \delta^{(k-1)}_0,\phi\ra \ln t,
\end{equation}
so the homogeneity is partly lost.\\
\\
For $\textup{Re}\, a>-1$ and $k\in\N$ the functions $x\mapsto x^a_\pm \ln^k x,$ defined on $\R$ by

\begin{equation*}
  x^a_+\ln^k x:= \begin{cases} x^a\ln^k x &\mbox{if } x>0 \\
  0 & \mbox{if } x\leq 0,
\end{cases} \;\;\; \text{ and } \;\;  \;
  x^a_-\ln^k x:= \begin{cases} |x|^a\ln^k |x| &\mbox{if } x<0 \\
  0 & \mbox{if } x\geq  0,
\end{cases}
\end{equation*}
 are locally integrable,  and hence  they can be extended as  tempered distributions on $\R$ for any
 $a\in\C\setminus\Z_{<0}.$ We can build $x^a_+\ln^l x$ from differentiating $x^a_+$  with respect to $a$.
 For $\textup{Re}\, a>-k$ , the map $a\mapsto x^a_+$ is  analytic hence we have
\begin{equation}\label{Eq.log-derivatives}
 x^a\ln^l x =\partial_a^l x^a_+=\partial^l_a\Big( \frac{\partial^k_x x^{a+k}}{(a+k)\cdots (a+1)} \Big).
\end{equation}

\begin{remark}
 The extensions are unique by the uniqueness of the analytic continuation.
\end{remark} 

\begin{prop}\label{Thm.3.2.1.HI}({\rm \cite{H}, Theorem 3.1.11.})
 Let $I$ be an open interval on $\R$ and let
 $$Z=\{ z\in\C\ |\ Re(z)\in I,\ 0<\textup{Im}\, z <\gamma\}$$ be a one sided complex neighborhood. 
 For an analytic function $f$ in $Z$ such that for a non-negative integer $N,$ $$ |f(z)|\leq C(\textup{Im}\, z)^{-N},\ z\in Z,$$
 then $f(\cdot+iy )$ has a limit $f_0\in \mathcal{D}^{\prime, N+1} (I)$ as $y\rightarrow 0$, that is,
 $$\lim_{y\rightarrow 0^+}\int f(x+iy)\phi (x)dx=\la f_0,\phi \ra,\ \phi\in C^{N+1}_0(I).$$
\end{prop}

The following equations which contain double signs in either side, are to be
understood as double equations: one equation holding for the upper signs and the other holding for the lower signs. By  Proposition \ref{Thm.3.2.1.HI}, the function $z^a$, defined in $\C\setminus\R_-$ as $e^{a\ln z},$ where $z\in\R_+$, has distributional boundary values\index{distributional boundary values}
\begin{equation*}
 (\xi\pm i0)^a:=\dlim_{\eta\rightarrow 0^+}(\xi\pm i\eta)^a,
\end{equation*}
 on the real axis  from the upper and lower half planes.
 Now, for any test function $\phi$, the function $a\mapsto \la (\xi\pm i0)^a,\phi \ra$ is the limit of entire analytic functions, so it is
 an entire analytic function. Additionally, on the one hand, for $a\in\C\setminus \Z_{<0},$ we have
\begin{equation}\label{3.2.9.HI}
 (\xi\pm i0)^a=\xi^a_++e^{\pm a\pi i}\xi_-^a,
 \hspace{1cm}  \xi (\xi\pm i0)^a=\xi_-^{a+1}+e^{\pm a\pi i} \xi_+^{a+1}.
\end{equation}
Furthermore, if $a=-k$ where $k$ is a positive integer, we have
\begin{equation*}
 (\xi\pm i0)^{-k}=\xi^{-k}_+ +(-1)^k \xi_-^{-k} \pm \frac{\pi i (-1)^k}{(k-1)!}\delta_0^{(k-1)}.
\end{equation*}
On the other hand, if we set $\underline{\xi}^{-k}:=\big((\xi+i0)^{-k}+(\xi-i0)^{-k}\big)/2$,  we have
\begin{equation*}
 \underline{\xi}^{-k}=\xi^{-k}_+ +(-1)^k \xi^{-k}_-,
 \hspace{1cm}  \xi\underline{\xi}^{-k}=\underline{\xi}^{1-k}.
\end{equation*}
Finally,  for all $\phi\in C^1_0(\R),$ we have $\underline{\xi}^{-1}(\phi)=( \rm{P.V.} \frac{1}{\xi})(\phi)
 :=\lim_{\epsilon\rightarrow 0}\int_{|\xi|>\epsilon}\frac{\phi(\xi)}{\xi}d\xi,$
 which is usually called the {\em principal value} of $\frac{1}{\xi}$.\\
 \\
\textbf{The Fourier transform of log-homogeneous distribution:} Now, we describe some results related with Fourier transform of distributions in (\ref{Eq:Homog.Distrib}) and (\ref{Eq.log-derivatives}), in some case we give proofs for such results since they are not easy to found in the literature.\\ 
\\
The  Fourier transform $\mathcal{F}:\mathcal{S}(\R)\to \mathcal{S}(\R)$  can be extended can be extended to an isomorphism $\mathcal{F}:\mathcal{S}'(\R)\to \mathcal{S}'(\R)$ given by $\la \mathcal{F}[f], u  \ra:=\la  f,\mathcal{F}[u]  \ra$ for $u\in \mathcal{S}(\R)$. The following definition of the Laplace transform is due to L. Schwartz \cite{Schwartz52}.\\
For a compactly supported distribution $u\in\mathcal{E}'(\R^n)$, the {\em Laplace-Fourier transform}, by abuse of notation we also denote by $\mathcal{F}[u]$, is defined by a $C^\infty$-function given by 
\begin{equation}\label{Eq.analytic.Fourier.compact.dist}
\xi\mapsto  \mathcal{F}[u](\xi):=\la u,e^{-ix\xi}\ra,
\end{equation}
here $\la\,\,,\,\,\ra$ is the dual pair. The Laplace-Fourier transform may be extended to tempered distribution $u$ in $\mathcal{S}'(\R)$ whose support is bounded at the left, then in this case  \begin{equation}\label{Eq.analytic.Fourier}
\mathcal{F}[u](\zeta):=\la u, e^{-i x\zeta}\ra_{\mathcal{S}^\prime,\mathcal{S}},
\end{equation}
is well defined on $\textup{Re}\, \zeta<0$. Moreover, setting $\zeta=\xi+i\eta$ and fixed $\xi$, the Laplace-Fourier transform $\mathcal{F}[u](\zeta)$ coincides  with the Fourier transform of the function $\widehat{u}(\eta)$.

\begin{lemma}{\em (\cite{H}, Theorem 7.1.16.)}
 If $u\in\mathcal{S}^\prime(\R)$ is a homogeneous distribution of degree $a$, then its Fourier transform
 $\mathcal{F}[u](\phi):=\la u,\mathcal{F}[\phi]\ra$ is  a homogeneous distribution of degree $-a-1$.
\end{lemma}

\begin{remark}\label{Remark.Indep.Cut-off}
 Let $\omega_1$ and $\omega_2$ be cut-off functions. We have that $\omega_1-\omega_2$ has compact support
 contained in a ring. Thus, the Fourier transform $(\omega_1-\omega_2)x^a_+$ is a smooth function on $\R$ and
 $$\mathcal{F}[(\omega_1-\omega_2)x^a_+]=O([\xi]^{-\infty}).$$ Similarly for the inverse Fourier transform.
\end{remark}

Consequently, the subsequent statements of  Proposition \ref{Fourier.transform.homogeneous}, Proposition \ref{Inverse.Fourier.transform.homogeneous} and  Corollary \ref{Fourier.transform.log-homogeneous}, are independent of the chosen cut-off function $\omega$ modulo smooth function of order $O([\xi]^{-\infty}).$

\begin{prop}\label{Fourier.transform.homogeneous}
 Let $\omega$ be a cut-off function and  $a\in\C\setminus \Z_{<0}$.
\begin{enumerate}
\item  The Fourier transform of $x^a_\pm$
 is positively homogeneous of degree $-a-1$ and
\begin{equation*}
 \mathcal{F}[x^a_\pm](\xi)=\frac{\Gamma(a+1)}{\sqrt{2\pi}}e^{\mp i\pi(a+1)/2}(\xi\mp i0)^{-a-1},
\end{equation*}
 and from Equation (\ref{3.2.9.HI}) it follows that $ \mathcal{F}[(x\pm i0)^a](\xi)=\frac{\sqrt{2\pi}}{\Gamma(-a)}e^{\pm i\pi (a/2)} \xi^{-a-1}_\pm.$

\item The Fourier transform of $(1-\omega) x^a_+$ is a smooth function on $\R$ and we have
\begin{equation*}
 \mathcal{F}[(1-\omega)x^a_\pm](\xi)=O([\xi]^{-\infty}) \text{ as } |\xi|\rightarrow \infty.
\end{equation*}

\item
 Since $\mathcal{F}[x^a_+]=\mathcal{F}[\omega  x^a_+] +\mathcal{F}[(1-\omega) x^a_+]$  it  follows  that
\begin{equation*}
 \mathcal{F}[\omega  x^a_\pm](\xi)=\frac{\Gamma(a+1)}{\sqrt{2\pi}}e^{\mp i\pi(a+1)/2}(\xi\mp i0)^{-a-1}+O([\xi]^{-\infty}).
\end{equation*}

\item Finally,
\begin{equation*}
 \frac{d}{d\xi}\mathcal{F}[x^a_\pm](\xi)=-i \mathcal{F}[x(x^{a}_\pm)](\xi)=-i \mathcal{F}[x^{a+1}_\pm](\xi).
\end{equation*}

\end{enumerate}
\end{prop}

\begin{proof}

\begin{enumerate}
\item
 Observe that when $\eta>0$ and $\textup{Re}\, a>-1$ the Fourier transform
 of the rapidly decreasing function $e^{-\eta x}x^a_+$ is
\begin{equation*}
\xi\mapsto \Gamma(a+1)\int_0^\infty
x^ae^{-x(\eta+i\xi)}dx=\Gamma(a+1)(\eta+i\xi)^{-a-1}\int_0^\infty
z^ae^{-z}dz
\end{equation*}
 where the last integral is taken on the ray generated by $\eta+i\xi$ and
 $z^a$ is defined in $\C$ slit along $\R_-$ (so $1^a=1$). From the Cauchy integral formula it follows that the integral can be taken along $\R_+.$
 Therefore the Fourier transform  reads
\begin{equation*}
 \xi\mapsto \Gamma(a+1)(\eta+i\xi)^{-a-1}=\Gamma(a+1)e^{-i\pi(a+1)/2}(\xi-i\eta)^{-a-1}.
\end{equation*}
 When $\eta\rightarrow 0$, the Fourier transform of $x^a_+$ has the form
\begin{equation*}
 \la x^a_+,\hat{\phi} \ra=\Gamma(a+1)e^{-i\pi (a+1)/2}\la (\xi-i0)^{-a-1},\phi\ra,\ \phi\in\mathcal{S},\ Re(a)>-1.
\end{equation*}
 Both sides are entire  analytic functions of $a$ so   the identity extends to all $a\in\C$. The second statement follows by the
 well known identity $\Gamma(-a)\Gamma(a+1)=\pi/ \sin (-\pi a).$

\item Since $\omega(x)x^a_+$ is a compactly supported distribution, it follows from Lemma \ref{Lemma.Fourier.2} that the Fourier transform
 of $\omega(x)x^a_+$ is a smooth function on $\R$. This proves the first part of this item.
 For the second part, let $\phi\in\mathcal{S}(\R).$
\begin{eqnarray*}
 \la \xi^k\mathcal{F}[(1-\omega)x^a_+], \phi \ra &=&\la \mathcal{F}[(1-\omega)x^a_+], \xi^k\phi\ra\\
 &=& \la (1-\omega)x^a_+, \mathcal{F}[\xi^k\phi]\ra\\
 &=& \la (1-\omega)x^a_+, (-D^k)\mathcal{F}[\phi]\ra\\
 &=& \la \mathcal{F}[D^k((1-\omega)x^a_+)], \phi \ra.
\end{eqnarray*}
 If $k<\textup{Re}\, a+1$, then $D^k \big( (1-\omega)x^a_+\big)$ is an integrable function and its Fourier transform  is therefore well-defined and bounded.
 Thus we see that $ \mathcal{F}[(1-\omega)x^a_\pm](\xi)=O([\xi]^{-k})$ for each $k\in\N$.
\item Item (3) in Proposition \ref{Fourier.transform.homogeneous} is immediate  from (1) and (2).
\item It follows from the definition of the derivative of a distribution.
\end{enumerate}
\end{proof}

\begin{prop}\label{Inverse.Fourier.transform.homogeneous}
  Let $\omega$ be a cut-off function and $a\in\C.$
\begin{enumerate}
\item The inverse Fourier transform of $(\xi\pm i0)^a$ is
\begin{equation*}
 \mathcal{F}^{-1}[(\xi\mp i0)^a]= \frac{\sqrt{2\pi}}{\Gamma(-a)}e^{\pm  i\pi (a/2)} x^{-a-1}_\pm.
\end{equation*}
 In particular, $\mathcal{F}^{-1}[\mathcal{F}[x_\pm^a]]= x_\pm^a$.
\item The inverse Fourier transform of the tempered distribution $\omega  (\xi\mp i0)^a$ is a smooth function
 and we have $$ \mathcal{F}^{-1}[\omega (\xi\mp i0)^a](x)=O([x]^{-\infty}).$$
\item
 Since $\mathcal{F}^{-1}[(\xi\mp i0)^a]=\mathcal{F}^{-1}[\omega (\xi\mp i0)^a] +\mathcal{F}^{-1}[(1-\omega)  (\xi\mp i0)^a]$  it follows
\begin{equation*}
 \mathcal{F}^{-1}[(1-\omega) (\xi\mp i0)^a](x)=\frac{\sqrt{2\pi}}{\Gamma(a+1)}e^{\pm  i\pi a/2} x^{-a-1}_\mp+O([x]^{-\infty}).
\end{equation*}
\item Finally,
\begin{equation*}
 \frac{d}{d x}\mathcal{F}^{-1}[(\xi\mp i0)^a](x)= -i \mathcal{F}^{-1}[(\xi\mp i0)^{a+1}](x).
\end{equation*}
\end{enumerate}
\end{prop}

\begin{proof}
 The first statement follows from the first item in the above Proposition by the Fourier transform inversion formula $\mathcal{F}^{-1}=\mathcal{F}^{3}$:
\begin{eqnarray*}
 \mathcal{F}^{2}[(\xi\pm i0)^a]  &=& \frac{\sqrt{2\pi}}{\Gamma(-a)}e^{\pm i\pi (a/2)} \mathcal{F}[x_\pm^{-a-1}]\\
 &=& \frac{\sqrt{2\pi}}{\Gamma(-a)}e^{\pm i\pi (a/2)} \frac{\Gamma(-a)}{\sqrt{2\pi}}e^{\pm i\pi(a/2)}(\xi\mp i0)^{a}\\
 &=& e^{\pm i\pi a}(\xi\mp i0)^{a}.
\end{eqnarray*}
 Then $\mathcal{F}^{3}[(\xi\pm i0)^a]=e^{\pm i\pi a}\mathcal{F}[(\xi\mp i0)^{a}]=\frac{\sqrt{2\pi}}{\Gamma(-a)}e^{\pm i\pi (a/2)}x_\mp^{-a-1}.$
 The others statement follow in a similar way.
\end{proof}

As a consequence of Proposition \ref{Fourier.transform.homogeneous} and (\ref{Eq.log-derivatives}) we have the following result.

\begin{cor}\label{Fourier.transform.log-homogeneous}
 Let $a\in\C\setminus \Z_{<0},\ k\in \N$ and let $\omega$ be a cut-off function.
\begin{enumerate}
\item   There exists some $k\in\Z$ such that $\textup{Re}\, a+k>-1$ and
\begin{equation}\label{Fourier.defin.log}
  \mathcal{F}[x^a_\pm \ln^l x_\pm ](\xi)=i^k\xi^k \partial^l_a\Big( \frac{\mathcal{F}[x^{a+k}_\pm](\xi)}{(a+1)\cdots (a+k)}  \Big)
\end{equation}
 is independent of the choice of integer $k$. In particular, $\mathcal{F}[x^a_\pm \ln^k x_\pm ](\xi)$
 is log-polyhomogeneous of homogeneity degree $-a-1$ and logarithmic degree $l$, i.e. for any $a\in\C\setminus \Z_{<0},\ l\in\N,$
 there exist constants $c^+_{ij}, c^-_{ij},\ j=0,1,\cdots l$ (depending of $a,l$) such that
\[ \mathcal{F}[x^a_\pm \ln^l x_\pm ](\xi)= \sum\limits_{j=0}^{l}c^+_{j}\, \xi^{-a-1}_+\ln^{l-j} \xi+ \sum\limits_{j=0}^{l}  e^{\mp(-a-1)\pi i}c^+_{j}\, \xi^{-a-1}_-\ln^{l-j} \xi.\]
In particular for $l=0$, we have
\[ \mathcal{F}[x^a_+ ](\xi)= c^+ \xi^{-a-1}_++e^{i(a+1)\pi }c^+\xi^{-a-1}_-.\]
\item The Fourier transform of $(1-\omega) x^a_+\ln^l x$ is a smooth function on $\R$ and as $|\xi|\rightarrow \infty$
$$ \mathcal{F}[(1-\omega) x^a_\pm \ln^l x_\pm ](\xi)=O([\xi]^{-\infty}).$$
\item Finally, there exist constants $c^+_{ij}, c^-_{ij},\ j=0,1,\cdots l$ (depending on $a,l$) such that
\[ \mathcal{F}[\omega  x^a_\pm \ln^l x_\pm ](\xi)=\sum\limits_{j=0}^{l}c^+_{j}\, \xi^{-a-1}_+\ln^{l-j} \xi+\sum\limits_{j=0}^{l}c^-_{j}\, \xi^{-a-1}_-\ln^{l-j} \xi
 +O([\xi]^{-\infty}).\]
\end{enumerate}
\end{cor}

\begin{proof}
\begin{enumerate}
\item
 For $\textup{Re}\, a>-k,$ it follows that
\begin{eqnarray*}
 \mathcal{F}[x^a_\pm \ln^l x ](\xi)= \partial^l_a\Big( \frac{\mathcal{F}[\partial^l_x x^{a+k}_\pm](\xi)}{(a+1)\cdots (a+k)}\Big)
  =i^k\xi^k \partial^l_a\Big( \frac{\mathcal{F}[x^{a+k}_\pm](\xi)}{(a+1)\cdots (a+k)}\Big).
\end{eqnarray*}
 The independence of the choice of $k$ comes from the independence of $k$ in the definition $x^a_\pm$ in (\ref{Eq.3.Hormander}).
It follows from  Proposition \ref{Fourier.transform.homogeneous}  that $\mathcal{F}[x^a_+]$ is positively homogeneous of degree $-a-1.$
 Thus, $\mathcal{F}[x^a_\pm \ln^l x_\pm ](\xi)$ is log-polyhomogeneous of homogeneity degree $-(a+k)-1+k=-a-1$ and logarithmic degree $l$.
 More precisely, by  a straightforward computation we obtain
\begin{eqnarray*}
 \mathcal{F}[x^a_+\ln x]&=& i^k\xi^k \partial^l_a\Big( \frac{\mathcal{F}[x^{a+k}_+](\xi)}{(a+1)\cdots (a+k)}\Big) \\
  &=& (2\pi)^{1/2}i^k\xi^k \partial^l_a\Big( \frac{\Gamma(a+k+1)e^{-i\pi (a+k+1)/2}}{(a+1)\cdots (a+k)}  (\xi-i0)^{-a-k-1}\Big)\\
  &=&  \frac{(2\pi)^{1/2}i^k \Gamma(a+k+1)e^{-i\pi (a+k+1)/2}}{(a+1)\cdots (a+k)} (\xi-i0)^{-a-1}\\
  &&\ \ \ \     \times \Big(-\sum\limits_{j=1}^{k}\frac{1}{a+j}-\frac{\Gamma'(a+k+1)}{\Gamma(a+k+1)}+\frac{i\pi}{2}+\ln (\xi-i0) \Big).
\end{eqnarray*}
The assertion comes from $ (\xi\pm i0)^a=\xi^a_++e^{\pm a\pi i}\xi_-^a$, see (\ref{3.2.9.HI}). The general assertion for $\mathcal{F}[x^a_+\ln^l x]$ follows from $\mathcal{F}[x^a_+\ln^l x]=\partial^{l-1}_a \mathcal{F}[x^a_+\ln x]$.

\item  Follows from item $2.$ in Proposition \ref{Fourier.transform.homogeneous} and equation (\ref{Fourier.defin.log}).
\item
 The assertion follows from item $1.$ and $2.$ in Corollary \ref{Fourier.transform.log-homogeneous}. and
 $$x^a_\pm \ln^l x=\omega x^a_\pm \ln^l x +(1-\omega )x^a_\pm \ln^l x.$$
\end{enumerate}
\end{proof}

\subsection{ A Paley-Wiener theorem for functions with log- polyhomogeneous asymptotic behavior at zero and
positive support }\label{SS:P-W.Fourier.Log-poly} 

We establish a Paley-Wiener type theorem, for the Fourier transform, of  functions with log-polyhomogeneous asymptotic behavior at zero (cf. Definition \ref{Def.Conormal.Sing.0}).\\
\\
Notice that, for $u\in L^2(\R_\pm)$, the extension by zero of $e_\pm  u$ lies in $L^2(\R)$. Similarly, we can extend the operators $e_\pm$ to functions $u\in\mathcal{S}_{\PP}(\R_+)$ for any appropriate power set $\PP$. The following Proposition is due to C. Neira, E. Schrohe and S. Paycha.
 
\begin{prop}\label{extension.poly}
 Let $\PP\in\uP$ be an appropriate power set.  Any  $u$ in $\mathcal{S}_{{\PP}}(\R_+)$ admits an extension
  $ e_+u  $ in $ \mathcal{S}^\prime(\overline{\R}_+)$, where
 $ \mathcal{S}^\prime(\overline{\R}_+)$ denotes the set of tempered  distributions on $\R$ with support in $\overline{\R}_+$,
 and  we have,  as $x\rightarrow 0^+$,
\begin{equation}\label{Conormal.Sing}
 e_+u(x)\sim\sum\limits_{j=0}^{\infty}\sum\limits_{k=0}^{m_j}a_{jk}x^{p_j}_+\ln^k x.
\end{equation}
\end{prop}

\begin{proof}
 Given  a cut-off function $\omega$, there is a  sequence  $\{a_{jk}\}$  of complex numbers such that
 $\exists N\in \N$,
\begin{equation*}
 \omega(x)\big(u-\sum\limits_{j=0}^{N}\sum\limits_{k=0}^{m_j}a_{jk}x^{p_j}\ln^k x\big)=O(x^{\textup{Re}\, p_{N+1}  }).
\end{equation*}
 may be extended by zero for $N$ large enough. Moreover, equation (\ref{Eq.3.Hormander}) yields the extension $x^a_+\ln x$  of $x^a\ln x$ to $\R$ by zero. Finally, we obtain
 $$e_+u(x)=\omega(x) \big( \sum\limits_{j=0}^{N} c_jx^{d+j}_+ +O(x_+^{\textup{Re}\, p_{N+1} }  ) \big) +(1-\omega)(x)u(x).$$
\end{proof}

Let $\PP$ be an \textit{appropriate} power sets. For $u\in \mathcal{S}_{\PP}(\R_+)$, $e_+u$ defines a functions on $\R$ with support in $\overline{\R}_+.$ Applying Proposition \ref{extension.poly} we can set the following definition.

\begin{defin}\label{def.gen.W-H}
Let $\PP\in \uP$ be an appropriate power sets and let \index{$\mathcal{H}^{\pm}_{\PP}$}
\[
 \mathcal{H}^{\pm}_{\PP}=\{ \mathcal{F}[e_\pm u](\xi\mp i0)\ |\ u\in \mathcal{S}_{\PP}(\R_\pm)\}.
\]
Let us  further set $\mathcal{H}_\PP=\mathcal{H}^{+}_\PP\oplus \mathcal{H}^{-}_\PP$.
\end{defin}

\begin{lemma}\label{lemma:polyhomHd}

Taking Fourier transform in Lemma \ref{Lemma.subsets}, it follows that for any appropriate power sets $\PP=\{(p_j,m_j)\}\in\uP$, we have
\[
 \mathcal{H}^\pm=\mathcal{H}^{\pm}_{\PP_0} \subset \mathcal{H}^{\pm}_\PP\subset \mathcal{H}^{\prime,\pm},
\]
where $\PP_0=\{(j,0)\ |\ j=0,1,\cdots \}.$
 Moreover, we have
$$\mathcal{H}^{\pm}_\PP\subset C^\infty_{\PP'}(\R)$$  where
 $\PP'\in\oP$ is a convenient power sets, which contains $\{(-p_j-1, m_j-i)\ |\ i,j=0,1,\cdots\}$ as subset.
\end{lemma}

\begin{proof}
It just remain to show
$ u\in \mathcal{S}_\PP (\R_\pm)\Longrightarrow \mathcal{F}[e_\pm u]\in C^\infty_{\PP'}(\R)$
 for some power sets $\PP'$ containing $\{(-p_j-1, m_j-i)\ |\ i,j\in\N\}$ as subset.
 Let $\PP$ be an appropriate power sets and let $\omega$ be a cut-off function and $u\in \mathcal{S}_\PP(\R_+).$ There exists $N\in\N$ such that
\begin{equation*}
 u(x)=\omega \big( \sum\limits_{j=0}^{N}\sum\limits_{k=0}^{m_j}a_j x^{p_j}\ln^k |x|+(1-\omega)u(x)+ O(x^{ \textup{Re}\, p_{N+1} }) \big).
\end{equation*}
 Therefore, the Fourier transform of $u$ is
\begin{equation*}
 \mathcal{F}[e_+u](\xi)=\sum\limits_{j=0}^{N}\sum\limits_{k=0}^{m_j}a_j \mathcal{F}[\omega x^{p_j}_+\ln^k |x|]
 +\mathcal{F}[(1-\omega)u(x)]+\mathcal{F}[O(x^{ \textup{Re}\, p_{N+1}   }_+)].
\end{equation*}
The conclusion follows for combining the following facts:
\begin{enumerate}
 \item The term $\mathcal{F}[(1-\omega)u]$ is a continuous and rapidly decreasing function since $(1-\omega)u$ is a Schwartz function.
\item
 $\mathcal{F}[x_\pm^{p_j}\ln^k x]=\mathcal{F}[e_\pm x_\pm^{p_j}\ln^k x]$ is well-defined since $p_j\notin \Z_{<0}\ j\in\N$.
\item
 From Corollary \ref{Fourier.transform.log-homogeneous}, we have
\begin{eqnarray}\label{Eq:Fourier.eq}
 \mathcal{F}[\omega  x^{p_j}_\pm \ln^k |x|](\xi)&=&
\sum\limits_{j=0}^{k}c^+_{j} \xi^{-p_j-1}_+\ln^{k-j}| \xi|+\sum\limits_{j=0}^{k}c^-_{j} \xi^{-p_j-1}_-\ln^{k-j} |\xi |\nonumber\\
&&+O([\xi]^{-\infty}).
\end{eqnarray}
\item Finally, the Fourier transform of the term $O(x^{ \textup{Re}\, p_{N+1}     })$ is a function of the form $O([\xi]^{ -\textup{Re}\, p_{N+1} -1})$.
\end{enumerate}
 From remark \ref{Remark.Indep.Cut-off}, $\mathcal{F}[\omega x_\pm^{p_j}\ln^k x]$
 is actually independent of the choice of the cut-off function $\omega$ modulo continuous rapidly decreasing functions.
\end{proof}


\begin{example}
 Let $d\in \C\setminus \Z_{<0}$. If $u(x)$ denotes the function defined as $x^d e^{-x}$ if $x>0$ and $0$ if $x\leq 0$ we
 have $u\in \mathcal{S}_P(\R_+)$ for $\PP=\{(d+j,0)\ |\ j\in\N\},$ more precisely,
\[ u(x)\sim\sum_{j=0}^{\infty} (-1)^j x^{d+j}\text{ as } x\rightarrow 0^+.\]
 Hence there exists a sequence a complex number $a_0,a_1,\cdots ,$ such that $$\mathcal{F}[u](\xi)\sim\sum_{j=0}^\infty a_j (\xi-i0)^{-d-j-1}\; \text{ as }\, |\xi|\rightarrow \infty.$$
\end{example}

\begin{prop}\label{prop.analytic.extension}
 Any function in $ \mathcal{H}^{\pm}_\PP$ for  an appropriate power set $\PP\in\uP$, can be represented by an analytic function on $\C_{\mp}$
 that extends continuously to $\overline{\C}_\mp$ in $\mathcal{S}^\prime(\R)$.
\end{prop}

\begin{proof}
 Let $\PP$ be an appropriate power sets and $\omega$ be a cut-off function and $u\in \mathcal{S}_\PP(\R_\pm)$. There exists $N\in\N$ such that
\begin{equation}\label{eq:decomposition}
 \mathcal{F}[e_\pm  u](\xi)=\sum\limits_{j=0}^{N}\sum\limits_{k=0}^{m_j}a_j \mathcal{F}[\omega x^{p_j}_\pm \ln^k| x|]+\mathcal{F}[O(x^{ \textup{Re}\, p_{N+1}     })]
 +\mathcal{F}[(1-\omega)u(x)].
\end{equation}
 We  want to prove that each term in (\ref{eq:decomposition}) can be represented by an analytic function on $\C_{\mp}$. We first note that from Remark \ref{Remark.Indep.Cut-off} it follows that the assertion is independent (modulo analytic function) of the cut-off function used in the proof. Next, we see that the proof follows from the following observations:
\begin{enumerate}
\item
 The assertion is valid for term $\mathcal{F}[\omega x^{p_j}_\pm\ln^k x]$
 since $e_\pm( x^{p_j} \ln^k x)=x^{p_j}_\pm \ln^k x\in\mathcal{S}^\prime(\R)$ has support in $[0,\infty)$
 (resp.$(-\infty, 0]$) and Equation (\ref{Eq.analytic.Fourier}).
\item
 The assertion for the second and third terms follows from Proposition \ref{Paley-Wiener.thm.one.sided} and the
 fact that the Fourier transform of $(1-\omega)(x)u(x)$ has order $O([\xi]^{-\infty})$, respectively.
\end{enumerate}
\end{proof}

We are now ready to characterize functions in $\mathcal{H}^{\prime,\pm}_\PP$, for appropriate power sets $\PP\in\uP$. Recall  $\mathcal{H}^\pm=\mathcal{H}^{\pm}_{\PP_0} \subset \mathcal{H}^{\pm}_\PP\subset \mathcal{H}^{\prime,\pm},$ where $\PP_0=\{(j,0)\ |\ j=0,1,\dots \}.$ 

\begin{thm}\label{Type Paley-Wiener.thm's}
 Let $\PP\in\uP$ be an appropriate power set.  A function $h(\xi)$ lies in $  \mathcal{H}^{+}_\PP$ (resp. $  \mathcal{H}^{-}_\PP$) if and only if
\begin{itemize}
\item it can be represented by an analytic  function $h$ on $\C_{-} $ (resp. on $\C_{+}$) that extends in $\mathcal{S}^\prime(\R)$ continuously to
 $\overline{\C}_-$ (resp.  $\overline{\C}_+$)  by $h(\xi- i 0)$ (resp. $h(\xi+ i 0)$),
\item $h$ has the following growth at infinity $\forall \zeta=\xi\mp    i\eta\in \C_{\mp},\exists  m\in\N, \exists C>0$ such that
 \begin{equation*}
 |h(\zeta)|\leq C_n  \la  \zeta\ra^m 
\end{equation*}
 \item    Moreover, if $\PP$ does not contain logarithmic indices, i.e. $\PP$ has the form $\{(p_j,0)\, |\, j\in\N\}$, then  $h(\xi- i 0)$ has an asymptotic expansion as $|\xi|\rightarrow\infty$
\begin{eqnarray}\label{Eq:Transm.symmetry.cond}
h(\xi- i 0)\sim \sum\limits_{j=0}^{\infty}c_{j} (\xi- i0)^{-p_j-1} =\sum\limits_{i=0}^{\infty}   c^+_j\,  \xi^{-p_j-1}_++c^-_j\, \xi_-^{-p_j-1},
\end{eqnarray}
where $c_j$ are described in (\ref{eq:decomposition.2}), $c^-_j=e^{i (p_j+1)\pi}c^+_j$ for $j=0,1,\dots$, and
$$h(\xi+ i 0)\sim \sum\limits_{j=0}^{\infty}c_{j} (\xi+ i0)^{-p_j-1} =\sum\limits_{i=0}^{\infty}   c^+_j\,  \xi^{-p_j-1}_++c^-_j\, \xi_-^{-p_j-1},$$
with  $c^-_j=e^{-i (p_j+1)\pi}c^+_j$ for $j=0,1,\dots$.
{More  generally}, allowing logarithmic powers $\PP=\{(p_j,m_j\}$, the asymptotic expansion of $h(\xi\mp i 0)$ reads
\begin{eqnarray*}
h(\xi\mp i 0)&\sim& \sum\limits_{i,j=0}^{\infty}\sum\limits_{k=0}^{m_j}a_{ijk} (\xi\mp i0)^{-p_j-1} \ln^{k-i} (\xi\mp i0)\\
&=& \sum\limits_{i,j=0}^{\infty}\sum\limits_{k=0}^{m_j}   a^{+}_{ijk}\xi^{-p_j-1}_+ \ln^{k-i} (\xi\mp i0)       +b^{+}_{ijk}\xi^{-p_j-1}_- \ln^{k-i} (\xi\mp i0)
\end{eqnarray*}
for some complex numbers $b^+_{ijk}$ an $b^-_{ijk}$ for $i,j=0,1,\cdots$ and $k=0,1,\cdots m_j$.
\end{itemize}
\end{thm}
\begin{proof}
 For $u\in \mathcal{S}_\PP(\R_+)$ set $H=\mathcal{F}[e_\pm  u]\in \mathcal{H}^{\pm}_\PP$. 
 We known from (\ref{Eq.analytic.Fourier}) $H$ can be represented as
 an  analytic function $h$ on $\C_\mp$, and that extends continuously to $\overline{\C}_\mp$. By the maximum principle it follows that $|h(\zeta)|\leq C_n \la \zeta\ra^m$ for $|\zeta|>1$ and $\textup{Im}\, \zeta\geq 0$. First, let us consider asymptotic expansion of $u$ with not log terms. Now, to prove that $H$ has an asymptotic expansion at infinity (\ref{Eq:Transm.symmetry.cond}), let  $u\sim \sum a_jx _+^{p_j}\in \mathcal{S}_{\PP}(\R_+)$, i.e.  Let  $\omega$ be a cut-off function and $u\in \mathcal{S}_\PP(\R_\pm)$. There exists $N\in\N$ such that
\begin{equation*}
 u(x)=\omega \big( \sum\limits_{j=0}^{N}\sum\limits_{k=0}^{m_j}a_j x^{p_j}+ O(x^{ \textup{Re}\, p_{N+1}     }) \big)+(1-\omega)u(x).
\end{equation*}
It follows from Proposition \ref{Fourier.transform.homogeneous} and $ (\xi\pm i0)^a=\xi^a_++e^{\pm a\pi i}\xi_-^a$  that
\begin{eqnarray}\label{eq:decomposition.2}
 \mathcal{F}[e_+  u](\xi)&=&\sum\limits_{j=0}^{N}a_j \mathcal{F}[\omega x^{p_j}_+ ]+\mathcal{F}[O(x^{ \textup{Re}\, p_{N+1}     })] +\mathcal{F}[(1-\omega)u(x)] \nonumber\\
 &=&\sum\limits_{j=0}^{N}a_j\, \frac{\Gamma(p_j+1)}{\sqrt{2\pi}}e^{- i\pi(p_j+1)/2}(\xi- i0)^{-p_j-1}\\
 &&+O(x^{- \textup{Re}\, p_{N+1}-1     })+O([\xi]^{-\infty})\nonumber\\
&\sim&\sum\limits_{j=0}^{N}a_j\, \frac{\Gamma(p_j+1)}{\sqrt{2\pi}}e^{- i\pi(p_j+1)/2}(\xi^{-p_j-1}_++e^{i (p_j+1)\pi}\xi_-^{-p_j-1}).\nonumber
\end{eqnarray}
Let us prove the converse assertion. Let $h$ be an  analytic function which satisfies the  enumerated  properties in Theorem \ref{Type Paley-Wiener.thm's}. 
From Proposition \ref{Inverse.Fourier.transform.homogeneous}, we have $H$ lies in $\mathcal{H}^{\pm}_\PP$ for some $\PP\in\uP$.
Now, we consider the following asymptotic expansion
\begin{equation*}
 u(x)=\omega \big( \sum\limits_{j=0}^{N}\sum\limits_{k=0}^{m_j}a_j x^{p_j}\ln^k x+ O(x^{ \textup{Re}\, p_{N+1}     }) \big)+(1-\omega)u(x).
\end{equation*}
The corresponding assertion for logarithmic terms follows from (\ref{Eq:Fourier.eq}).
\end{proof}


Now, from Theorem \ref{Type Paley-Wiener.thm's} we deduce well-known Paley-Wiener type theorems concerning to the Fourier transform of rapidly decreasing smooth function with positive support which are smooth up to the boundary (see \cite{RS}, Section 2.1.1.1. Corollary 3; also  \cite{BdM}). We recall that \cite{RS} the space $\mathcal{S}(\overline{\R}_+)$ is a Fr\'echet space equipped with the countable family of semi-norms, for $m,k\in\N$, $ \|u\|_{m,k}:=\sup_{x\in\R_+}|x^k\partial^m u(x)|$.  If we set
\begin{equation*}
 \mathcal{H}^{\pm}:=\mathcal{H}^{\pm}_{\PP_0}=\{  \mathcal{F}[e_\pm u](\xi)\ |\ u\in \mathcal{S}(\overline{\R}_\pm)\}.
\end{equation*}
The range $\mathcal{H}^+$ equipped with the image topology of $\mathcal{S}(\overline{\R}_+)$, i.e. the countable family of semi-norms given by
\begin{equation*}
 \|| h_+\||_{m,k}=\|\xi^k\partial^m_{\xi}h_+(\xi)\|_{L^2}, \text{ for }m,k\in\N,
\end{equation*}
is a Frech\'et space, and we have the following result.
 
\begin{lemma}
  The map
\begin{equation*}
 \mathcal{F}:e_+\mathcal{S}(\overline{\R}_+)\rightarrow \mathcal{H}^+
\end{equation*}
 is an isomorphism of Fr\'echet spaces. We denote by $\mathcal{F}^{-1}$ its inverse.
\end{lemma}

\begin{proof}
 The bijectivity follows from bijectivity of the Fourier transform acting over the space $L^2$. To see that $\mathcal{F}$ is continuous it suffices to show that $h_n(\xi)\rightarrow  h(\xi)$ as ${n\rightarrow \infty}$ in the Fr\'echet topology of $\mathcal{H}^+$ with $h_n(\xi):=\mathcal{F}[e_+u_n](\xi)$, and $h(\xi):=\mathcal{F}[e_+u](\xi)$ where $(u_n)$  is a sequence in $\mathcal{S}(\overline{\R}_+)$  which converges to $u$ in the Fr\'echet topology of $\mathcal{S}(\overline{\R}_+)$.   
We have, for any $m,k\in\N$, that as  $n\rightarrow \infty$, 
\begin{eqnarray*}
  \|| h_n-h\||_{m,k}= \|\xi^m\partial^k_\xi(\mathcal{F}[u_n]-\mathcal{F}[u](\xi)])\|_{L^2}
  =\| \partial^m_x \big( x^k(u_n(x)-u(x))\big) \|_{L_2}\rightarrow 0.
\end{eqnarray*}
\end{proof}
 
As a consequences of the dominated convergence theorem, we have that for any 
integrable function $f(s,x)$ defined for $(s,x)\in  I\times\R$, where $I$ is an open interval of $\R$, and there exists $\partial f /\partial s$ and $g$ integrable  such that $|\partial f(s,x)/\partial x|\leq g(x)$. Then $F(s)=\int f(s,x)dx$ is differentiable and
$dF/ds=\int df/dsdx$.

\begin{lemma}\label{Paley-WienerProposition} 
 If $u\in \mathcal{S}(\overline{\R}_+)$ (resp. $u\in \mathcal{S}(\overline{\R}_-)$) then $\mathcal{F}[e_\pm u]$ is a
 smooth function which can be represented by an entire function $h(\zeta)$ in $\C_-$
 (resp. $\C_+$) which  extends  continuously to $\overline{\C}_-$ (resp. $\overline{\C}_+$).
 By abuse of notation we shall set $h(\zeta)=\mathcal{F}[e_\pm u](\zeta)$.
\end{lemma}

\begin{proof}
 Let $u\in \mathcal{S}(\overline{\R}_+)$. For $\zeta=\xi-i\eta$ with $\xi\in\R$, $\eta<0$. The integral (cf. (\ref{Eq.analytic.Fourier}))
\begin{equation*}
  h(\zeta)=\int_0^\infty e^{ix(\xi-i\eta)}u(x)\normd\, x=\int_0^\infty e^{ix\xi} (e^{x\eta}u(x))\normd\, x
\end{equation*}
 converges absolutely.
Since $u$ is a rapidly decreasing function, it follows $x\mapsto xu(x)$ is $L^1$-integrable, and using the above comment we can   differentiate in  the integral. The function  $h$ is therefore an analytic function as a consequence of the Cauchy-Riemann equations.
\end{proof}

\begin{prop}\label{Paley-Wiener.thm.one.sided} 
 The space $\mathcal{H}^{\pm}$ consists precisely of smooth functions $h$ defined on $\R$ which
\begin{itemize}
\item[(i)] can be represented by an analytic  function $h$ on $\C_{\mp} $ that extends continuously to  $\overline{\C}_\mp$,
\item[(ii)] and such $h$ has the following asymptotic expansion at infinity
\begin{equation}\label{E.Sch.2.1}
 h(\zeta)\sim\sum\limits_{k=0}^{\infty} a_k\zeta^{-k-1}\text{ for } | \zeta| \rightarrow\infty,\  \zeta\in\C_\pm,
\end{equation}
 which can be differentiated formally.
\end{itemize}
Moreover, setting $\mathcal{H}:=\mathcal{H}^+\oplus \mathcal{H}^-\oplus \mathcal{P}$,
where $\mathcal{P}$  is the set of all polynomials, we have
\begin{itemize}
\item[(iii)]
 The space $\mathcal{H}^\pm$ and $\mathcal{H}$ are algebras with respect of the product of complex valued functions.
\item[(iv)]
 $\mathcal{H}$ consists precisely of all functions $h\in C^\infty(\R)$ which have an
 expansion $h\sim \sum\limits_{k=0}^{\infty}c_k\xi^{-k-1}$ for $|\xi|\rightarrow \infty$ in $\R$, which can be
 differentiated formally. 
 \end{itemize}
\end{prop}

\begin{proof}
The first part follows as a consequence of Proposition \ref{Paley-WienerProposition}. The second part is proved as  follows, for  $u\in\mathcal{S}(\overline{\R}_+)$, integration by parts we get $\int e^{-ix\zeta}u(x)\,\normd x\sim \sum_{j=0} u^{(j)}(0)/(i\zeta)^{j+1}$, and it can be extended on $\C_-$ setting $\zeta=\xi+i\eta$. For the last part  concerning the symbolic property of $h(\xi\mp i0)$,  let $k\in \N$ and $h=\mathcal{F}[e_\pm   u]$.
Then $x\mapsto x^k u$ also lies in ${\mathcal S}(\overline{\R}_+)$ and we have
\[(-i)^k  \mathcal{F}[e_\pm x^k u](\xi)=  \partial_\xi^k   \mathcal{F}[e_\pm   u](\xi)=   \partial_\xi^k h(\xi) .\]
Hence
\[ \vert  \mathcal{F}[e_\pm x^k u]\vert= |\partial_\xi^k h(\xi )|\leq C_\eta \langle \xi\rangle^{-1-k}\quad\forall \xi\in \R, \]
so that $h$ defines a symbol of order at most $-1$. The item (iii) is obvious. It follows immediately from Theorem \ref{Type Paley-Wiener.thm's} item (iv).
\end{proof}


\subsection{A Paley-Wiener theorem for (tempered) distributions with
positive support}\label{SS:WP-thm.tempred.distributions}

The previous results extend to a similar results for tempered distributions. Before to describe these, let us first give  a Paley-Wiener results concerning to compact supported distribution in $\R$, see Equation \ref{Eq.analytic.Fourier.compact.dist}.

\begin{lemma}\label{Lemma.Fourier.2}{\em (\cite{Rudin}, Theorem 7.2. and 7.23)}
 For any compactly supported distribution $u\in\mathcal{E}'(\R^n)$ its {\em (Laplace-)Fourier transform} (\ref{Eq.analytic.Fourier.compact.dist})  can be extended to an analytical function with $\zeta=\xi+i\eta$ in the complex plane $\C$, cf. Equation (\ref{Eq.analytic.Fourier}). By abuse of notation we denote by $\mathcal{F}[u](\zeta)$ the analytical extension. Moreover, there exits constants $C,M>0$ and an $n_0$ such that, for every $\zeta\in\C^n$, 
\begin{equation}\label{Eq:P-Wv1}
|\mathcal{F}[u](\zeta)|\leq C\la \zeta\ra^{n_0}e^{M|\textup{Im}\, \zeta|}.
\end{equation}
Conversely, any entire function $h$ satisfying (\ref{Eq:P-Wv1}) in $\C$, there exists $u\in \mathcal{E}'(\R)$ such that $\mathcal{F}[u](\zeta)=h(\zeta)$ for every $\zeta\in\C$. 
\end{lemma}

Let us set
\[
 \mathcal{H}^{\prime, \pm}=\{  \mathcal{F}[U](\xi)\ |\ U\in \mathcal{S}^\prime(\overline{\R}_\pm)\},
\]
 where $\mathcal{S}^\prime(\overline{\R}_\pm)$ stands for the set of tempered distributions with positive/negative support. 
 Set
 $$\mathcal{H}'=\mathcal{H}^{\prime,+}\oplus \mathcal{H}^{\prime,-}.$$

Before we state the following Theorem, let us recall a well-known fact \cite{RS},  which can give the flavor of the following result: A function $H\in L^2(\R_\xi)$ is a Fourier transform  of a function $u$ in $L^2(\R_+)$  if and only if $H$ can be represented by an analytic function $h(\zeta)$ on $\C_{\mp}$ that extends continuously to $\overline{\C}_\mp$ by $h(\xi+i0)$ in $L^2$, i.e. $$ h(\xi\pm i\eta)\longrightarrow H_\xi\text{ in } L^2(\R),\text{ as } \eta\rightarrow 0^+.$$
 Moreover, by imposing the condition that $h$ satisfies the estimative $|h(\zeta)|\leq Ce^{a|\zeta|}$ for some $C,a>0$, we have $\textup{supp}\,  u\subset [-a,a]$. For tempered distributions in $\mathcal{H}^{\prime,\pm}$ we find:

\begin{prop}\label{WP-thm.tempred.distributions}{\em (\cite{SZ}, Theorem 1.)}
A tempered distribution $H$ on $\R$  lies in $  \mathcal{H}^{\prime,\pm}$ if and only if

\begin{itemize}
 \item[(i)] it can be represented by an analytic function $h$ on $\C_{\mp}$,  that extends continuously to  $\overline{\C}_\mp $ by\footnote{ $h(\xi\mp i 0)$ denotes the limit $h(\xi+i\eta)\longrightarrow H_\xi \text{ in } \mathcal{S}^\prime(\R),\text{ as } \eta\rightarrow 0^\mp.$ } $h(\xi\mp i0)$ in $\mathcal{S}^\prime(\R)$,

 \item[(ii)] and $h$ has the following growth at infinity $\forall \zeta=\xi+i\eta\in \C_{\mp},\exists m,n\in \N, \exists\, C_n>0,$ such that
  \begin{equation}\label{exponential.type.growth1}
    |h(\zeta)|\leq C_n\la \zeta\ra^m M_n(\eta),
  \end{equation}
  where $M_n(\eta)$ is the map equal to one for $\vert \eta\vert>1$ and $\vert \eta\vert^{-n}$ for $\vert \eta\vert\leq 1$.
\end{itemize}
In this case, the analytic representation $h$ of $H$ is given by $h=\mathcal{F}[U]$ the Laplace-Fourier transform.
\end{prop}

From Proposition \ref{WP-thm.tempred.distributions} one can equip a product to the space  $\mathcal{H}^{\prime,+}$ with an algebraic structure, see \cite{SZ}.

\begin{defin}
 Let $U$ and $V$ be distributions in $\mathcal{S}^\prime(\R)$ which can be extended to analytic functions $u$ and $v$  defined on
 $\C_\mp$ (cf. Proposition \ref{WP-thm.tempred.distributions}). We set
\begin{equation*}
 (UV)(\xi):= \lim_{\eta\rightarrow 0^\mp} u(\xi\pm   i\eta)v(\xi\pm i\eta)\text{ in } \mathcal{S}^\prime(\R).
\end{equation*}
\end{defin}

\begin{example}
 For $m\in\N$ we denote by $H^m(x)$ the function $x^m$, if $x\geq  0$ and $0$, if $x<0$. Its Fourier
 transform
\begin{equation*}
 \mathcal{F}[H^m](\xi)=m!(-i)^{m+1}\xi^{-m-1}-i^m\pi \delta^{(m)},
\end{equation*}
extends analytically to $h^m(\zeta)=(-i)^{k+1}k!\zeta^{-k-1}$   exists by Proposition \ref{WP-thm.tempred.distributions}.
 For $m,n\in\N$ we have
\begin{eqnarray*}
 (\mathcal{F}[H^m]\mathcal{F}[H^n])(\xi)&=&\lim_{\eta\rightarrow 0^-}h^m(\xi+i\eta)h^n(\xi+i\eta)\\
                                        &=& m!n!(-i)^{m+n+1}( -i\xi^{-m-n-2}-\pi \delta^{(m+n)}).
\end{eqnarray*}
\end{example}

Finally, from Proposition \ref{Thm.3.2.1.HI} and Proposition \ref{WP-thm.tempred.distributions}  we obtain a characterization of elements in $\mathcal{H}'$.

\begin{cor}
 A tempered distribution $H$ on $\R$ lies in $\mathcal{H}'$ if and only if
 it can be represented by an analytic function $h$ on $\C_+\cup\C_-$ and  it defines a distribution in $\mathcal{S}^\prime(\C)$ given by
\begin{equation*}
 H(\varphi)=\int\int h(\xi+i\eta)\varphi(\xi,\eta)d\xi d\eta,\ \varphi\in \mathcal{S}(\R^2),
\end{equation*}
such that
\begin{equation*}
 \la\  \frac{\partial H}{\partial \bar{\zeta}} ,\varphi\   \ra=\frac{i}{2} \la h(\cdot+i0)-h(\cdot-i0),\varphi(\cdot,0)\ra,
 \ \varphi\in \mathcal{S}(\R^2),
\end{equation*}
 and we have $H=h$ in $\mathcal{S}(\C\setminus\R)$.
\end{cor}

\end{document}